\documentclass[12pt]{amsart}
\usepackage{amsmath,
amssymb,amscd}
\usepackage{amsfonts}
\usepackage{latexsym}
\usepackage{epsfig}

\numberwithin{equation}{section}

\def\w{\omega}
\def\C{\mathbb{C}}
\def\r{r}
\def\cQ{\mathcal{Q}}
\def\rank{r}
\def\cE{\mathcal{E}}
\def\cM{\mathcal{M}}
\def\length{\operatorname{length}}
\def\ch{\operatorname{ch}}
\def\td{\operatorname{td}}
\def\Gr{\operatorname{Gr}}
\def\Spl{\operatorname{Spl}}
\def\O{\mathcal{O}}

\def\E{\mathcal{E}}

\def\P{\mathbb{P}}

\def\Z{\mathbb{Z}}
\def\Q{\mathbb{Q}}
\def\a{\alpha}

\def\l{\mathit{l}}
\def\O{\mathcal{O}}
\def\t{\tau}
\def\s{\sigma}

\def\ra{\rightarrow}

\def\H{\operatorname{H}}

\def\det{\operatorname{det}}

\def\PGL{\operatorname{PGL}}

\def\H{\operatorname{H}}

\def\M{\mathcal{M}}

\def\E{\mathcal{E}}

\def\Hom{\operatorname{Hom}}

\def\Ext{\operatorname{Ext}}

\def\End{\operatorname{End}}
\def\Pic{\operatorname{Pic}}

\def\Br{\operatorname{Br}}

\def\cal{\mathcal}

\newtheorem{lemma}{Lemma}[section]
\newtheorem{corollary}[lemma]{Corollary}
\newtheorem{theorem}[lemma]{Theorem}

\newtheorem{proposition}[lemma]{Proposition}

\theoremstyle{remark}
\newtheorem{remark}[lemma]{Remark}
\theoremstyle{definition}
\newtheorem{definition}[lemma]{Definition}

\begin{document}

\title[Rank two sheaves on K3 surfaces]{Rank two sheaves on K3 surfaces: A special construction}

\author{Colin Ingalls}
\address{Department of Mathematics and Statistics, University of New Brunswick,
Fredericton, New Brunswick,
Canada.}\email{cingalls@unb.ca}
\author{Madeeha Khalid}
\address{Dept. of Mathematics,
St. Patrick's College, Drumcondra Dublin 9, Ireland}
\email{madeeha.khalid@spd.dcu.ie}

\thanks{
The first author was supported by an NSERC Discovery Grant.
The second author was supported by the IRCSET Embark Initiative Postdoctoral
Fellowship Scheme, Ireland.}

\begin{abstract}Let $X$ be a K3 surface of degree 8 in $\P^5$ with hyperplane section $H$.  We associate to
it another K3 surface $M$ which is a double cover of $\P^2$ ramified on a sextic
curve $C$.  In the generic case when $X$ is smooth and a complete intersection of three quadrics, there is a natural correspondence between $M$ and the moduli space $\cM$ of
rank two vector bundles on $X$ with Chern classes $c_1=H$ and $c_2=4.$
We build on previous work of Mukai and others, 
 giving conditions and examples where $\cM$ is fine, compact, 
non-empty; and birational or isomorphic to $M$.  We also present an explicit calculation of the Fourier-Mukai transform
when $X$ contains a line and has Picard number two.\end{abstract}

\maketitle

\begin{section}{Introduction}

K3 surfaces are a special class of two dimensional complex manifolds which are of interest to both mathematicians and physicists.  They are compact, simply connected complex surfaces with trivial first Chern class. 
The geometry of moduli spaces of sheaves on K3 surfaces is also a very rich 
and intriguing area of study.  Much of the foundation of the subject was laid by Mukai 
in \cite{Mukai1}, \cite{Mukai2} and \cite{Mukai3}, where he derived some beautiful relations between Hodge structures of K3 surfaces and their associated
moduli spaces of sheaves.
In this paper we present some new results on moduli 
spaces of rank two sheaves as well as generalisations of previously
known ones.  One of our main results involves a calculation
of the cohomological Fourier-Mukai transform.  It is the first known such 
calculation for non-elliptically fibred K3 surfaces.

Let $X$ be a K3 surface of degree
8 in $\P^5$, given by the complete intersection 
of
three independent quadrics $Q_0, Q_1, Q_2.$  let $H$ denote its hyperplane class in ${\O_X}(1).$  
Associated to $X$ is a K3 surface $M$ which is a double cover of $\P^2$ ramified on a sextic.
Let $\cal{Q}$ be the net of quadrics spanned by $Q_0,Q_1,Q_2,$ and
 $C:=V(\det \cal{Q})$ the plane sextic curve parameterising the degenerate
quadrics.  
Let $\phi : M \rightarrow \P^2$ be the 
double cover of 
$\P^2$ branched along $C$. Then $M$ is also a K3 surface.  

If the rank of the degenerate quadrics in $\cal{Q}$ is
always $5$ then $C$ is smooth and hence $M$ is smooth. 
 It turns out that
each point of $M$ 
corresponds to a rank $2$ spinor bundle on $X$ with invariants $c_1 = H, c_2 = 4.$  So 
$M$ is naturally associated to a moduli space $\M$ of rank $2$ sheaves on $X.$  This classical 
correspondence is described in \cite{Mukai3}~Example~2.2.  However the complete proof is not presented.
We give the detailed proofs in sections 3.1.2 and 3.1.3.
 
%

In Section 4 we extend this result to arbitrary K3 surfaces of degree $8$ in $\P^5$
specifying the exact
conditions necessary for $\cM$ to be compact.  We prove the following
(Theorem ~\ref{thm:M-nonempty}, 
Proposition~\ref{prop:M-noncompact} and Theorem~\ref{thm:M-bir}).
\begin{theorem}\label{thm:intro}
Let $X$ be a K3 surface of degree 8 in $\P^5$.  Let $H$ denote its hyperplane
class and let $M$ be the associated double cover as described above.   Let $\cM$ 
be the moduli space of stable (with respect to $H$)
rank 2 sheaves on 
$X$ with Chern classes $c_1=H$ and $c_2=4$.  Then,
\begin{enumerate}
\item The moduli space $\cM$ is non-empty.
\item The moduli space $\cM$ is compact if and only if 
$X$ does not contain an irreducible curve $f$ such that $f^2=0$ and $f.H=4$.
Moreover, in this case every element of $\M$ is locally free.
\item If $X$ is a complete intersection then $\cM$ is birational to $M$.
\end{enumerate}
\end{theorem}
For $X$ a generic K3 surface of degree 8, $\Pic(X) = \Z H$ and $\M$
 is {\it not a fine moduli} space, i.e. there does not exist 
a universal sheaf $\cE$ on $X\times \M$ such that $\cE_{|_{X\times m}} \simeq E_m$
where $E_m$ is the isomorphism class of sheaves corresponding to $m$ in $\cM$.   However a {\it quasi-universal sheaf}  does exist.
An explicit construction of a quasi-universal sheaf is in \cite{Khalid}.  

In general $X$ is not a moduli space of sheaves on $M$.  The 
obstruction to the existence of a universal sheaf $\cE$, is an element $\alpha$ of $\Br(M)$.  In other words, an $\alpha$-twisted sheaf exists on $X \times M$.  So $X$ can 
be viewed as a moduli space of $\alpha$-{\it twisted} sheaves on $M$.  See for instance \cite{Caldararu}, \cite{Sawon}.  The order of $\alpha$ in $\Br(M)$ is 2 and it corresponds
to an even sublattice of $T_M$ of index 2.  For a discussion of two-torsion elements of 
$\Br(M)$ see \cite{Geemen}.

If $X$ contains a line, then it turns out the $\M$ is a fine moduli space and a universal 
sheaf $\E$ exists on $X \times M$.   In fact
$M \simeq X$ follows as an application of a more general result in  \cite{MadonnaNikulin}. 
  We find a lattice in $\H^*(X, \Z) / (\Z \cdot (2,\O_X(1), 2))$ which is 
isomorphic to $\Pic(M).$  Using Mukai's previous results and techniques from deformation theory, we
compute the Fourier-Mukai map on cohomology (Theorem 5.3)
$$f_{\E} : \H^*(X, \Z) \rightarrow \H^*(M, \Z).$$

We calculate explicitly the invariants of the universal sheaf $\E$ on 
$X \times M \simeq X \times X$ and prove that the inverse
Fourier-Mukai map presents $X$ as a moduli space of sheaves on $M$ also with the 
same Chern invariants (Theorem~\ref{thm:main}). 
\begin{theorem}
Let $X$ be a K3 surface of degree 8 in $\P^5$.  Assume that $X$ contains a line
and $\rho(X)=2$.  let $\M$ be the moduli space of rank $2$ sheaves on $X$ with 
 $c_1=H$ and $c_2=4.$  Then 
 \begin{enumerate}
\item  The moduli space $\M$ is a fine moduli space and is isomorphic to X.  
\item The universal sheaf $\E$ on $X \times X$ is symmetric, i.e. its restriction to 
either of the 
 two factors is a rank $2$ sheaf with Chern invariants $c_1 = H$ and $c_2 = 4.$
\end{enumerate} 
\end{theorem}

 \begin{remark}
Note that even if the moduli space is isomorphic to the original K3 surface 
it is
 not always the case 
that the universal sheaf of the moduli problem is symmetric.  For a counter example see 
5.3.8~\cite{HuybrechtsLehn}.
\end{remark}

The arguments for the above results use various techniques involving stability 
and our expression for the Fourier-Mukai transform. 
For instance along the way we prove that every sheaf $E_m$ corresponding to $m \in \M$ is also $\mu$-stable (also known as slope-stable).

 When $\rho(X) >1 $ it does not necessarily follow that $\M$ is fine. 
 We work out an example where $\rho(X) = 2$, $\M$ is compact, 
but $\M$ is not fine (Proposition~\ref{tricase} and 
Theorem~\ref{thm:M-nonfine}). 
\begin{theorem}
Let $X$ be a K3 surface of degree 8 in $\P^5$ that contains a plane
conic and $\rho(X)=2$.  Then
\begin{enumerate}
\item The associated K3 surface $\phi: M \rightarrow \P^2$ is branched over a 
sextic which admits a tritangent.
\item The moduli space $\cM$ is isomorphic to $M$, and is not fine, i.e. there does not exist a universal sheaf $\E$ on $X \times \cM$.
\end{enumerate}
\end{theorem}

We also exhibit an example when $\M$ is non-empty but {\it not compact}.  
This is the case when $X$
contains a curve $f$ with $f^2=0$ and
$f.H=4$.  
The generic $\M$ also non fine.

{\bf Acknowledgements}
The second author would like to thank Professor Le De Trang for his support during 
the author's stay at the International Centre for Theoretical Physics (Trieste), and Bernd Kreussler for many helpful discussions.

\section{Preliminaries}\label{sec:prelim}

In this paper unless stated otherwise 
\begin{itemize}
\item The base field is $\C.$
\item a {\it surface} means a nonsingular compact connected $2$ dimensional complex-analytic manifold.
\end{itemize}

\subsection{K3 surfaces}

K3 surfaces are like $2$ dimensional generalisations of the one dimensional complex torus in that 
they admit a nowhere vanishing global holomorphic two form.  They provide fascinating examples in the
 context of algebraic, differential and arithmetic geometry and more recently also in mathematical physics. 
 In this paper we only concern ourselves with their algebro-geometric nature.  The following is a standard
 definition.

\begin{definition}\label{def:k3}
A {\it K3 surface} is a smooth, compact, complex, simply connected surface with trivial canonical bundle.
\end{definition}

\begin{remark} The reader may find other equivalent definitions of a K3 surface in the literature which 
assume K\"{a}hlerity.  For instance, sometimes a K3 surface is defined as; a compact K\"{a}hler 4 manifold 
with holonomy  group $\mbox{SU}(2)$.
Showing this equivalence however is a non-trivial exercise and uses some deep results.
\end{remark}

We work only with algebraic K3 surfaces, 
i.e. those which admit an embedding in some 
projective space $\P^n.$  Some examples are; double covers of $\P^2$ branched along a smooth
sextic curve, smooth quartic hypersurfaces in $\P^3$, complete intersections of a quadric and cubic 
hypersurface in $\P^4$, and triple intersection of quadric hypersurfaces in $\P^5.$  A dimension count shows
 that each of these forms a $19$ dimensional family.
In general for each $n$ there is a 19 dimensional moduli space of K3 surfaces occurring as normal surfaces 
of degree $2n-2$ in $\P^n$.  For $n>5$ they
are not complete intersections.

Another interesting class of K3 surfaces are {\it Kummer surfaces.}  They are obtained by taking the
quotient space of the canonical involution on a 2-dimensional complex torus $T$ and blowing up the 16 
singular points.  If the complex torus $T$ is non algebraic then the associated Kummer surfaces is also non algebraic.
 There is a $20$ dimensional universal family of K3 surfaces, the generic member of which is non algebraic 
and which contains a countable dense union of $19$ dimensional subsets parameterising the algebraic K3 surfaces.  Algebraic Kummer K3 surfaces form a dense subset of this universal family.

The cohomology  and Hodge decomposition of a K3 surface is completely determined.

\begin{theorem}
Let $X$ be a K3 surface.  Then:
\begin{itemize}
\item $\H^1(X, \Z) = \H^3(X, \Z) = 0$.
\item $\H^2(X, \Z)$ is a rank 22 lattice and, with the cup product, is isometric to 
$$L:=(-E_8)^2\oplus U^3.$$  Here $E_8$ is the root lattice of the exceptional lie algebra $\mathfrak{e}_8$ 
and $U$ represents
the standard unimodular hyperbolic lattice in the plane.  
\item $h^{0,1} = h^{1,0} = h^{1,2} = h^{2,1} = 0$.
\item $h^{2,0} = h^{0,2} = 1, h^{1,1} = 20$.
\item $\Pic X \simeq  \H^{1,1}(X, \Z) \simeq N_X,$ where $N_X$ is the N\'eron-Severi group of $X$.
\end{itemize}
\end{theorem}

The rank of the Picard group is denoted by $\rho(X)$.  It can be anything from $0$ to $20$.
When $\rho(X) = 20$ then $X$
is said to be an {\it exceptional } K3 surface.  An example is the Fermat quartic 
hypersurface in $\P^3$ given by $x_0^4 + x_1^4 + x_2^4 + x_3^4 = 0$, where
$x_0, \ldots , x_3$ are homogeneous coordinates on $\P^3$.

We set $L_{\C} := L \otimes {\C}$ with the pairing $( \, , )$ extended $\C$-bilinearly.
For $\w \in L_{\C}$ we denote by $[\w] \in  {\P(L_{\C})}$ the line $\C\w$,
and set $$\Omega := \{ [\w] \in \P(L_{\C}) \mid ( \w , \w ) = 0\mbox{ and }(\w, \overline{\w}) > 0\}.$$
This set $\Omega$ is referred to as the {\it period domain} of K3 surfaces.  It has 
dimension 20.

An isometry $\alpha : \H^2(X, \Z) \rightarrow L$ is called a {\it marking} and determines a line in $L_{\C}$ spanned 
by the $\alpha_{\C}$ image of the nowhere vanishing holomorphic 2-form $\w_X$.
The relations $( \w_X , \w_X ) = 0$ and $( \w_X, \overline{\w_X} ) > 0$ imply that this
line considered as a point of $\P(L_{\C})$, lies in $\Omega$. This point denoted $\a(\w_X)$ is called the
 {\it period point} of the marked K3 surface $X, \alpha.$

If we have a family $p: \boldsymbol{X} \rightarrow S$ of K3 surfaces together with an isomorphism 
$\a:R^2p_*\Z\rightarrow L_S$ where $L_S$ is the locally constant sheaf on $S$ with values in $L,$  then it is called 
a {\it marked family.}  The isomorphism $\a$ is called a {\it marking}  of the family.  For example K3 surfaces occurring as quartic hypersurfaces in $\P^3$ form a $19$ dimensional family.  The 
base (space) $S$ is isomorphic to the set of irreducible homogeneous quartic polynomials
in four variables, modulo the natural action of $\PGL(4)$. The space of irreducible homogeneous quartic polynomials(modulo scale invariance) in four variables has dimension 
$\binom{4+4-1}{4} - 1 = 35 -1 = 34$.  The dimension of 
$\PGL(4)$ is $16-1 = 15$, so $34-15 = 19.$  The moduli space of K3 surfaces realised as double covers of the plane is paramterised by the space of homogeneous degree six polynomials modulo the action of $\PGL(3)$.  The moduli space of degree six polynomials in three variables (modulo multiplication by constants) has dimension 
$\binom{6+3-1}{6}-1 = 28 - 1 = 27$.  The group $\PGL(3)$ has dimension $9-1 = 8$.  Therefore we get a $27 - 8 = 19$ dimensional family of double covers of the plane.  A similar dimension count shows that the family of degree $8$ K3 surfaces in $\P^5$ has dimension 19 (\cite{GH}).

Given a family 
$p:\boldsymbol{X} \rightarrow S$, if we restrict $p$ to a simply connected open subset of $U$ of $S,$ then $R^2 p_*\Z$ is locally trivial and 
admits a {\it marking.}  So marked families always exist. 
A {\it deformation} of a K3 surface $X$ parameterised by $S$ is a flat family $p: \boldsymbol{X} \rightarrow S$ 
along with a base point $s_0 \in S$ and an isomorphism from $X$ to the fibre $\boldsymbol{X}_{|_{s_0}} = X_{s_0}.$
The notion of a {\it marked deformation} of a K3 surface $X$ is similar.

To any marked family of K3 surfaces given by maps $p:\boldsymbol{X} \rightarrow S,$ and
  $\a:R^2p_*\Z \rightarrow L,$  we obtain an associated {\it period mapping} 
\begin{eqnarray*}
\t: S &\rightarrow& \Omega \\ 
s & \mapsto & \a_{\C}(\w_{X_s})
\end{eqnarray*}
where $\a_{\C}(\w_{X_s})$ is the period point of the fibre $X_s$ over $s \in S$ via the marking $\a.$ 

We mention some of the main theorems about K3 surfaces.  For details of proofs and further properties of K3 surfaces see \cite{BPV}, \cite{Beauville}, 
\cite{GH}.

\begin{theorem}\label{thm:local-torelli}[Local Torelli Theorem]
There is a universal deformation of any K3 surface $X.$  The base is smooth, of dimension $20$ and the 
period mapping is a local isomorphism at each point of the base.
\end{theorem}
Roughly the idea behind the proof is as follows.
Any compact complex manifold $X$ has a "local moduli space"(Kuranishi family) 
$p: \boldsymbol{X} \rightarrow S$ parameterising small deformations of $X,$ where $X \simeq \boldsymbol{X}_{|_{s_0}} = X_{s_0}$
 for some base point $s_0 \in S.$  The base $S$ has dimension $h^1(T_X)$ and is smooth if $\H^2(X, T_X) = 0.$  
For $X$ a K3 surface, by Serre duality $\H^2(X, T_X) = \H^0(X, \Omega^1_X)^\vee = 0.$  So there is no obstruction 
to deformations and the local moduli space is smooth.  Also from the Hodge decomposition it follows that
$h^1(T_X) = h^1(\Omega^1_X) = 20.$  So $S$ has dimension $20.$  The differential $d \t (s_0)$ of the associated
 {\it period mapping} $\t : S \rightarrow \Omega$ is  locally injective.  Since $S$ and $\Omega$ have the same
 dimension, 
$\t: S \rightarrow \Omega$ is a local isomorphism.

All K3 surfaces are diffeomorphic but not necessarily isomorphic.  The question as to when two K3 surfaces are 
isomorphic is very interesting and leads to a type of Torelli theorem.

We first define the notion of a Hodge isometry and then give a criterion for when two K3 surfaces may be isomorphic.

\begin{definition}\label{def:hodge}
Let $X$ and $Y$ be surfaces.  We say that an isometry
$\H^2(X, \Z) \rightarrow \H^2(Y, \Z)$ is a {\it Hodge isometry}
if its $\C$-linear extension preserves the Hodge decomposition.
\end{definition}

\begin{theorem}[Torelli theorem]
Two K3 surfaces $X$ and $Y$ are isomorphic if and only if $\H^2(X, \Z)$ and
$\H^2(Y, \Z)$ are Hodge isometric.
\end{theorem}

Any algebraic K3 surface is K\"{a}hler but in fact a much stronger result holds.

\begin{theorem}
Every K3 surface is K\"{a}hler.
\end{theorem}

The above result is used in the proof of the following theorem.  The proof however is quite involved.

\begin{theorem}[Surjectivity of period mappping] \label{thm:k3period}
Every point of $\Omega$ occurs as the period point of some marked K3 surface.
\end{theorem}


Since a K3 surface has trivial canonical bundle, the Riemann-Roch theorem takes on a  particularly simple form.
Let $D$ be a divisor on $X$, then by Serre duality $\H^2(\O(D)) \simeq \H^0(\O(-D)).$  For ease of notation we denote the line bundle $\O(D)$ also by $D.$  The Riemann-Roch theorem for line bundles on surfaces states that 
$$\chi(D) = h^0(D) -h^1(D) + h^0(-D) = 2 + (K_X.D+D^2)/2.$$
For K3 surfaces this simplifies to $$\chi(D) = 2 + D^2/2.$$

We now state some formulae which we refer to in the following sections.
Let $E$ be a coherent sheaf of rank $r$, with Chern classes $c_i$, on a surface $X$.  Then
\begin{equation}\label{eq:chern}
\ch E =  r + c_1 t + \frac{1}{2} (c_1^2-2c_2)t^2 
\end{equation}

\begin{equation}\label{eq:td}
 \td E = 1 + \frac{1}{2} c_1t +\frac{1}{12}(c_1^2+c_2) t^2
\end{equation}

and the Hirzebruch-Riemann-Roch formula reads
\begin{equation}\label{eq:HRR}
\chi(E) = \sum_{i} (-1)^i h^{i}(E) = [\ch E.\td_X]_2   
\end{equation}  
where the subscript $2$ denotes the degree 2 component.


\subsection{Foundations}

We use some results by Mukai on moduli spaces of sheaves on K3 surfaces (\cite{Mukai2} \cite{Mukai3}).  Two of the main theorems from \cite{Mukai2} we state here for the reader's convenience.
Let $\H^*(X,\Z)$ be the total cohomology lattice of $X$.  In 
\cite{Mukai2}, Mukai defines a symmetric bilinear form on $\H^*(X,\Z)$
which is often called the {\it Mukai pairing}.  Let 
$$a = (a^0, a^1,a^2) \in \H^*(X,\Z) \ \text{where} \ a^i \in \H^{2i}(X, \Z).$$
\begin{definition}\label{def:mukai-pairing}
The {\it Mukai pairing} on $\H^*(X, \Z)$ is defined as follows
$$(a,b) = a^1 b^1 -a^0 b^2 -a^2b^0$$ where $a^{i}b^{2-i}$ represents the cup product of $a^{i}$ with $b^{2-i}$.
\end{definition}

There is a natural weight-2 Hodge structure on the Mukai lattice
given by
\begin{eqnarray*}
\tilde{\H}^{2,0}(X,\C) & := & \H^{2,0}(X,\C) \\ 
\tilde{\H}^{0,2}(X,\C) & := & \H^{0,2}(X,\C) \\
\tilde{\H}^{1,1}(X,\C) & := &  \H^{0}(X,\C) \oplus \H^{1,1}(X,\C) \oplus \H^{4}(X,\C).\end{eqnarray*}
\begin{definition}
An element $v \in \tilde{\H}(X,\Z)$ is called a {\it Mukai vector}.  The
vector $v$ is called {\it primitive} if $v$ is not an integral multiple of any other 
element.  It is called {\it isotropic} if $(v,v)=0$.
\end{definition}

Let $E$ be a coherent sheaf on $X$.  The 
Chern character $\ch(E)$ is an element of $\H^*(X,\Z)$.

\begin{definition}\label{def:mukai-vector}
The Mukai vector associated to a sheaf $E$ on $X$ is given by 
$$v(E):=\ch(E).\sqrt{\td_X}.$$  It is an element of $\tilde{H}^{1,1}(X,\Z)$
of the form $v(E)= (v^0,v^1,v^2)$ where $v^0$ is the rank of $E$
at the generic point and $v^1$ is the first Chern class $c_1(E)$.  
\end{definition}
It follows from equations (\ref{eq:chern}) and (\ref{eq:td}) that
$$ v(E)^2 = \frac{c_1(E)^2}{2} -c_2(E) +r(E).$$
The Euler characteristic pairing for two coherent sheaves $E,F$ on $X$
is given by 
$$ \chi(E,F) = \sum (-1)^i \dim \Ext^i_X(E,F).$$
So the Riemann-Roch Theorem implies  
$$ \chi(E,F) = - (v(E),v(F)).$$

Next we introduce two notions of stability. 

\begin{definition}\label{def:stability}{\it Mumford-Takemoto Stability:}
Let $E$ be a torsion free coherent sheaf on a smooth projective variety
$X$.  Let $A$ be an ample line bundle and let
$$\mu_A(E) =\frac{c_1(E). A^{\dim X -1}}{\r(E)}$$ be the {\it slope}
of $E$.  Then $E$ is {\it $\mu$-stable} 
(respectively {\it $\mu$-semistable}) with respect to $A$, if
$ \mu_A(F) < \mu_A(E)$ (respectively $ \mu_A(F) \leq \mu_A(E)$)
for every proper coherent subsheaf $F$ of $E$.
\end{definition}

Equivalently $E$ is $\mu$-stable (respectively $\mu$-semistable),
 if $\mu_A(W) > \mu_A(E)$ (respectively  $\mu_A(W) \geq \mu_A(E)$)
for every torsion free coherent quotient $W$ of $E$.

\begin{definition}\label{def:g-stability}{\it Gieseker Stability:}
Let $E$ be a torsion free sheaf on $X$ of rank $r,$ and $A$ an
ample line bundle.  Define the {\rm normalised} Hilbert polynomial 
$$ P_{A,E}(n) = \frac{1}{r} \chi(E \otimes A^{\otimes n}).$$
Then $E$ is {\it Gieseker stable} (respectively semistable) with respect to $A$, if for all
coherent subsheaves $F$ of $E$, with $0 < \rank(F) < \rank(E)$,
we have $P_{A,W}(n) < P_{A,E}(n)$ (respectively $P_{A,W}(n) \leq P_{A,E}(n)$),
for all $n >> 0$.
\end{definition}

\begin{itemize}
\item It is important to note that the notion of stability depends on the choice 
of the ample class $A$.  However it makes no difference if one uses $A$ or $mA$ for 
some positive integer $m$.

\item In this paper, when a sheaf $E$ satisfies the conditions of Mumford-Takemoto-stability we 
say $E$ is $\mu$-stable, and when $E$
satisfies the conditions of Gieseker stability we say that $E$ is stable. 
\end{itemize}

\begin{definition}\label{def:simple}
A coherent sheaf $E$ is called {\it simple} if $\End E \simeq \C.$
\end{definition}

If $E$ is stable then $E$ is simple but the converse is not necessarily true.  We give an 
example based on 1.2.10~\cite{HuybrechtsLehn}.  Let $L_1$
and $L_2$ be two non isomorphic line bundles  on a curve of genus 2, with 
$\mu(L_1) = \mu(L_2)$.  Then since line bundles are stable, $\Hom(L_2,L_1) =0.$  
Using the Riemann-Roch formula we calculate the dimension of $\Ext^1(L_2,L_1)$:
$$ dim({\Ext^1}(L_2,L_1) = - {\chi}({L_2}^{\vee} \otimes L_1) = 1.$$
Therefore there is a non trivial extension $E$, 
$ 0 \rightarrow L_1 \rightarrow E \rightarrow L_2 \rightarrow 0.$
Clearly $E$ is semi-stable but not stable.  One can also show that $E$ is simple.

Suppose $E$ is a torsion free sheaf with positive rank.  Then we
have the following implications
$$ E \mbox{ is $\mu$-stable } \Rightarrow E \mbox{ is stable } 
\Rightarrow E \mbox{ is semistable } \Rightarrow E \mbox{ is $\mu$-semistable.}$$

Let $E$ be a semistable sheaf.  Then there is a filtration
$$E_{*}: 0=E_0 \subset E_1 \subset \ldots \subset E_n=E$$
such that every successive quotient $F_i = E_i /E_{i-1}$ is stable
 and has the same slope as $E.$  Let $v(F_i) = (r_i, l_i, s_i)$ then $s_i/r_i = s(E)/r(E).$
Such a filtration is known as the {\it JHS filtration of E.}

We follow Mukai's notation and denote by $M^\mu_X$
 (respectively $SM^\mu_X$) the set of all isomorphism classes of all
$\mu$-stable (respectively $\mu$-semistable) coherent sheaves on $X$.
The space
$M^\mu_X$ is an open subset of the moduli space $M_X$ of
stable (in Gieseker's sense) sheaves on $X$.

Let $v \in \tilde{\H}(X,\Z)$ be of Hodge type (1,1).
Let ${\Spl_X}(v)$ be the moduli space of {\it simple} sheaves on $X$ with Mukai vector $v.$
Then ${\Spl_X}(v)$ is a smooth manifold of dimension $(v,v) +2$ with a holomorphic symplectic 
two form (\cite{Mukai1}~Theorem~0.1 and Theorem~0.3). 
 Let $A$ be an ample line bundle on $X$ and let $\cM_A(v)$ be 
the moduli space of {\it stable sheaves} (with respect to A) on $X$ with Mukai vector
$v.$  Then
$\cM_A(v)$ is an open subset of ${\Spl_X}(v),$ hence, if non-empty, it is also smooth of dimension $(v,v)+2.$

Suppose $v$ is isotropic and primitive, then,  $\cM_A(v)$
is two dimensional, if it is non-empty.  Also in this case there exists a sheaf
$\cE$ on $X \times \cM_A(v)$ called a {\it quasi-universal sheaf} (Definition A.4,
Theorem A.5 \cite{Mukai2}), such that $\cE$ is flat over $X \times \cM_A(v)$
and $\cE_{|_{X \times m}} \simeq E_m^{\oplus \sigma}$ for every point $m \in 
\cM_A(v)$, where $E_m$ is representative of the isomorphism class of stable sheaves corresponding to $m$ in $\cM_A(v).$
The integer $\sigma$ does not depend on $m \in \cM_A(v)$ and is called the 
{\it similitude} of $\cE$.
  The smallest such $\sigma$ is given by 
$$\sigma_{{min}}
= \gcd\{ (w,v) \, | \, w \in \tilde{\H}^{1,1}(X,\Z) \}$$
where $v$ is the Mukai vector of $E$.  If $\sigma_{min} = 1$, then $\E$ is a 
universal sheaf and $\cM$ is a called a {\it fine moduli space}.  In general $\cM$ is not a 
fine moduli space.

Let $\cM = \cM_A(v)$, and
$\pi_X,$ $\pi_\cM$ the projections of $X \times \cM$ to $X$
and $\cM$ respectively.  Let 
$$Z_\cE = 
\left(\pi^*_X \sqrt{\td_X}.\ch(\cE^\vee).\pi^*_\cM\sqrt{\td_\cM}\right)/\sigma(\cE),$$
where $\td_X$ is the Todd class of $X.$
Then $Z_\cE$ is an algebraic cycle and induces a homomorphism which preserves
the Hodge structure
$$\begin{array}{lcll}
f_\cE : & \tilde{\H}(X,\Q) & \rightarrow & \tilde{\H}(\cM,\Q) \\
        & w & \mapsto & \pi_{M*} (Z_\cE . \pi_X^*(w)). 
\end{array}$$
 The map $f_\cE$ sends $v$ to the fundamental cocycle in $\H^4(\cM,\Z)$ 
and maps $v^\perp$ onto $\H^2(\cM,\Q) \oplus \H^4(\cM,\Q)$.  The following
 are results due to Mukai.

\begin{theorem}\label{thm:mukai}{(Theorems 1.4, 1.5, 4.9 \cite{Mukai2})}
Let $X$ be an algebraic K3 surface and $v$ a primitive isotropic vector
of $\tilde{\H}^{1,1}(X,\Z)$.  Assume that the moduli space $\cM_A(v)$ is 
non-empty and compact.  Then 
\begin{enumerate}
\item The moduli space $\M_A(v)$ is  irreducible and is a K3 surface.
\item A quasi-universal sheaf $\cE$ on $X \times \cM_A(v)$ 
exists and induces an isomorphism of Hodge structures 
$$f_\cE : \tilde{\H}(X,\Q) \rightarrow \tilde{\H}(\M_A(v),\Q)$$
that is independent of the choice of $\cE$.
\item The map
$$f_\cE: v^\perp/ \Z v \rightarrow \H^2(\cM_A(v),\Z)$$ is an isomorphism
of Hodge structures compatible with the bilinear forms on $v^\perp/\Z v$
and $\H^2(\cM_A(v),\Z)$.
\item If $\cM_A(v)$ is fine, i.e. $\sigma(\E)=1$, then $Z_\cE$ is integral 
and $f_\cE$ gives a Hodge isometry between the {\it lattices} 
$\tilde{\H}(X,\Z)$ and $\tilde{\H}(\cM,\Z)$.
\end{enumerate}
\end{theorem}

\section{A Special Construction}\label{sec:rk2vb}

In this section we present our results related to a classical example.
As before let $X$ be a K3 surface of degree 8
in $\P^5$ and let $H$ denote its hyperplane class in $\O_X(1)$.  Then
$X$ lies on three independent quadrics and in general is a 
complete intersection.  Let $\cQ$ be the net of quadrics spanned by $Q_0,Q_1,Q_2$.  
Let $C:=V(\det \cQ)$ denote the plane sextic curve parameterising the
degenerate quadrics in the net, and let $\phi:M \ra \P^2$ be the
double cover of $\P^2$ ramified along $C$.  If the rank of the
degenerate quadrics in $\cQ$ is always 5, then $C$ is smooth.  
Conversely,
given a smooth plane sextic curve and a choice of an ineffective 
{\it theta-characteristic} $L$ on $C$, there exists a family of quadrics
$\cQ$ in $\P^5$ such that $V(\det \cQ) = C$.  So there are as many nets of 
quadrics $\cQ$ as there are theta-characteristics $L$ on $C$ with $h^0(L) =0$,
(Theorem 1, \cite{Tjurin1}).  

Now suppose that the rank of the quadrics in $\cQ$ is bigger than or 
equal to 5 and that $\rho(X)=1$.  Let $\M = \M_H(2,H,2)$ be the moduli space of sheaves
on $X$ with Mukai vector $v=(2,H,2)$, stable with respect to $H$.  Equivalently 
$\M$ is the moduli space of stable sheaves on $X$ with $c_1 = H$ and $c_2 =4$.
It follows from Mukai's Theorem (Section 2.2, Theorem~2.7), that if $\M_H(2,H,2)$ is non-empty and
and compact, then it is an 
irreducible K3 surface.  Each point in $M$ naturally corresponds to a rank 2 vector bundle on $X$
with $c_1 = H$ and $c_2 = 4$, 
\cite{Mukai3}~Example~2.2.  
We prove below(Sections 3.1.2 and 3.1.3) that in fact each of these bundles is 
$\mu$-stable with respect to $H$ and that the natural morphism $M \rightarrow \cM$ is is an {\it isomorphism}. 
  
If there are singular quadrics in the net with rank strictly less than $5$, the sextic 
curve $C$ parameterising the degenerate curves singular and hence $M$ is singular.  In that 
case $\M$ is birational to $M$.

Recall that $\M$ is a fine moduli space if and only if 
$$1 = \s_{{min}} = \gcd\{(w,v) :w \in \tilde{\H}^{1,1}(X,\Z), v=(2,H,2) \}.$$
In general $\Pic(X)=\Z H$, hence any 
$w \in \tilde{\H}^{1,1}(X,\Z)$ is of the form
$\{(a,bH,c) | a,b,c \in \Z \}$.  So $(w,v) =8b-2a-2c = 2(4b-a-c)$ and hence $\s_{min} = 2$.  
This means that if $\rho(X)=1$, $\M$ is never a fine moduli space and hence a universal sheaf never exists for the generic case.
It is clear that in order to get a fine moduli space $\rho(X)>1$ is a necessary condition.  However it is not 
sufficient as we show in Theorem~5.5 that there exist $X$ and $M$ such that 
$\rho(X) = \rho(M) = 2$ and $\M \simeq M$ but $\M$ is not a fine moduli space.

\subsection{Quadrics}
Let $X$ be the complete intersection of three quadrics $Q_0, Q_1, Q_2$ such that the rank of the 
generic quadric in the net is $5$.
Let $Q$ be a smooth quadric in the net.
Then $Q$ is isomorphic to $\Gr(2,4),$ the Grassmannian of two dimensional
vector subspaces of $\C^4$, (equivalently the variety of lines in $\P^3$).
The homology of $Q$ is given by Schubert cycles $\s_1,\s_2,\s_{1,1},\s_{2,1}$.
The cycle $\s_1$ is given by the hyperplane sections of $Q,$ the cycles
$\s_2$ and $\s_{1,1}$ correspond to the two distinct families of projective
planes contained in $Q,$ and the cycle
$\s_{2,1}$ is given by lines in $Q$.  For 
details on Grassmanians, Chern classes of universal bundles, and Quadrics, see \cite{GH} Chapters 1, 3 and 6 respectively. 

For the reader's convenience we give below the intersection product on $\H_*(Q,\Z)$.
$$ \s_1^2 = \s_2 + \s_{1,1}$$
$$ \s_1 . \s_2 = \s_1 . \s_{1,1} = \s_{2,1}$$
$$ \s_2.\s_2 = \s_{1,1}.\s_{1,1} = \s_1 . \s_{2,1} =1$$
$$ \s_2. \s_{1,1} = 0.$$

\subsubsection{Spinor bundles}\label{subsec:spinor}
Let $Q \subset \P^5$ be a smooth four dimension quadric and let 
$$ 0 \rightarrow S \rightarrow \O_Q^4 \rightarrow F \rightarrow 0 $$
be the universal exact sequence given by choosing an isomorphism
$Q \simeq \Gr(2,4)$.  The bundles $S$ and $F^\vee$ are {\it spinor bundles} on
$Q$ since they arise
from spin representations of $SO(6, \C)$ \cite{Ottaviani}.  

The next fact is a combination of Theorems 2.3 and 2.8 of \cite{Ottaviani}.
\begin{proposition} \label{Ovanish}
Let $S,F^\vee$ be spinor bundles on $Q$, a smooth quadric of dimension 4, then
$$H^i(Q,S(t)) =0 \mbox{ for } 0<i<4, \mbox{ for all } t \in \Z$$
$$H^0(Q,S(t)) = 0 \mbox{ for all } t \leq 0, \quad \quad h^0(Q,S(1)) = 4,$$
and identical results for $F^\vee$ in place of $S$.  Also
$$ S^\vee \simeq S(1) \quad \quad F^\vee \simeq F(-1).$$
\end{proposition}
 
The homology class of $X$ in $\H^*(Q,\Z)$ is $(2\s_1)(2\s_1) = 4\s_1^2 = 4(\s_2 + \s_{1,1})$.
Let $\{\sigma_{i,j}^*\}$ denote the cohomology basis that is 
the Poincar\'e dual  of the homology basis
$\{\sigma_{i,j}\}$.  
Consider the bundles $S^\vee$ and $F.$  Then, as in the Gauss-Bonnet
Theorem I of \cite{GH} p. 410,
$$\rank(S^\vee)=\rank(F)=2,$$
$$c_1(S^\vee) = c_1(F) = \s_1^*$$
$$c_2(S^\vee) = \s_{1,1}^* \quad c_2(F) = \s_2^*.$$
So we see that 
$$c_1(S^\vee_{|_X}) = c_1(\O_X(1)) = c_1(F_{|_X})$$
$$c_2(S^\vee_{|_X}) = (\s_{1,1}^*,4(\s_2+\s_{1,1})) = 4$$
$$c_2(F^\vee_{|_X}) = (\s_2^*, 4(\s_2+\s_{1,1})) = 4.$$
Hence ${S^\vee_{|_X}}$ and $F_{|_X}$ are non-isomorphic bundles, both of which 
have Mukai vectors $(2,H,2)$. 

These bundles are in fact also $\mu$-stable as we show after some preliminary
results.  First we state a proposition which gives a criterion for determining
when two vector bundles are isomorphic.

\begin{proposition}
Let $V_1,V_2$ be stable vector bundles on a smooth projective variety.  
Suppose that
$V_1$ and $V_2$ have the same slope.  Then
$$h^0(V_1 \otimes V_2^\vee) = \left\{ \begin{array}{ll}
1 & \mbox{if } V_1 \simeq V_2 \\
0 & \mbox{if } V_1 \simeq\!\!\!\!\!{/} \,\, V_2. \end{array} \right.$$
\end{proposition}

 The proof is in \cite{Friedman} Chapter 4, Proposition 7,
or \cite{HuybrechtsLehn}, Propositions 1.2.7 and Corollary 1.2.8.

\begin{corollary}\label{OvanishU}
Let $Q,Q'$ be a pair of distinct smooth quadrics in $\P^5$.  Let $S$
be a spinor bundle on $Q$ and let $U=Q \cap Q'$.  Then
$$ H^0(U,S_{|_U}(t)) = 0 \mbox{ for } t \leq 0 \quad 
\quad h^0(U,S_{|_U}(1))=4$$
$$ H^2(U,S_{|_U}(t))=H^1(U,S_{|_U}(t))=0 \mbox{ for all } t$$
and identical results for $F^\vee$ in place of $S$. Also
$$ S^{\vee}_{|_U} \simeq S(1)_{|_U} \quad \quad F^{\vee}_{|_U} \simeq F(-1)_{|_U}.$$
\end{corollary}
\begin{proof}
This follows the cohomology long exact sequence associated to the short
exact sequence 
$$ 0 \rightarrow S(-2+t) \rightarrow S(t) \rightarrow S(t)_{|_U} \rightarrow 0.$$ and the above Proposition.

To show that $S^{\vee}_{|_U} \simeq S(1)_{|_U}$ we take the short exact sequence

$$ 0 \rightarrow S \rightarrow \O_Q^4 \rightarrow F \rightarrow 0,$$
and tensor with $S(-1).$  The associated long exact sequence in cohomology shows that
$h^0(S \otimes S(-1) = h^1(S \otimes S(-1))=0$.  

Then we consider the first short exact sequence again for $t = 0$ and tensor with $S(1)$.  The associated long exact 
sequence in cohomology shows that $h^0(U, S \otimes S(1)) = 1$. 
Since $U$ is a complete intersection.
$c_1(S^{\vee})_{|_U} = c_1(S(1))_{|_U}$, hence
$\mu(S^\vee)_{|_U} = \mu (S(1))_{|_U}$.  
It now follows from Proposition~3.2 above that $S^\vee_{|_U} \simeq S(1)_{|_U}$.
\end{proof}

We now prove that $S^{\vee}$ and $F$ are $\mu$-stable. 

\begin{corollary}  Consider a net generated by three
smooth quadrics $Q_1,Q_2,Q_3$ in $\P^5$ with base locus a complete intersection
surface $X$.  Let $S$ be a spinor bundle on a smooth quadric in the
net. 
\begin{enumerate}
\item Then 
$$ H^0(X,S_{|_X}(t)) = 0 \mbox{ for } t \leq 0 \quad 
\quad H^0(X,S_{|_X}(1))=4$$
$$ H^1(X,S_{|_X}(t))=0 \mbox{ for all } t$$
and identical results for $F^\vee$ in place of $S$

\item For the generic K3 surface $X$, $S^{\vee}_{|_X}$ and $F_{|_X}$ are $\mu$-stable 
with respect to the polarisation $\O_X(1)$.
\end{enumerate}
\end{corollary}

\begin{proof}
The proof of the first claim is similar to that of 
Corollary~\ref{OvanishU}.

We now prove part 2. As before let $H$ denote the hyperplane class in $\O_X(1)$.
  Let $L$ be a line bundle 
with an inclusion map into $S^\vee_{|_X}$.  For a generic K3 $X$ as above, 
$\Pic(X)$ has rank $1$
so $L$ is of the form $\O_X(kH)$ for some $k \in \Z$.
Since taking sections is left exact we 
get $h^0(L) \leq h^0(S^\vee_{|_X})=4$.  Now $\mu(S^\vee_{|_X}) = H.H/2=4$ and $\mu(L) = 8k$.
By the Riemann-Roch theorem for line bundles on a surface we get 
$$ \chi(L) = 2 + L^2/2 =2+4k^2.$$
If $k \geq 1$ then $L$ is effective and $h^2(L)=h^0(-L) =0$ so 
$$h^0(L) \geq 2+4k^2 \geq 6$$
Since $h^0(L) \leq 4$ we see that $k \leq 0$ so $\mu(L) \leq 0$ and therefore
$S^\vee_{|_X}$ is $\mu$-stable with respect to $H$.  It is enough to check sub-line bundles
by Lemma 5, Chapter 4 of \cite{Friedman}.
A similar argument proves $\mu$-stablity for $F$. 
\end{proof}

So far we have shown the existence of these bundles for smooth $Q$.  In 
fact even for $Q$ singular we get vector bundles with these invariants.
Let $Q$ be a quadric in the net and for each $x \in X$ consider
the variety $T_xQ \cap Q$ where $T_xQ$ is the projective tangent
space to $Q$ at $x$.  This is a singular quadric consisting of the lines
in $Q$ passing through $x$.  Let $\Gamma \simeq \P^3$ be a linear
space in $T_xQ$ disjoint
from $x$.  Then $T_xQ \cap Q$ is a cone over a quadric
surface $Q'$ in $\Gamma$.  The quadric $Q'$ is smooth if $Q$ is smooth,
and is
 a cone over a smooth plane conic if $Q$ is singular with rank 5.
\subsubsection{Smooth Quadrics}
  Consider first 
the case when $Q$ is smooth.  Then $Q'$ is smooth and contains two families 
of lines.  So $T_xQ \cap Q$ is spanned by two families of planes.  These 
correspond to one dimensional families of Schubert cycles $\sigma_2$ and 
$\sigma_{1,1}.$  
Given a point $p \in \P^3$ we write $\sigma_2(p)$ for the plane
in $G(2,4)$ parameterizing lines in $\P^3$ that contain $p$.
Given a hyperplane $h \subset \P^3$ we write $\sigma_{1,1}(h)$ for
the plane in $G(2,4)$ parameterizing lines in $\P^3$ that are
contained in $h$. 
Since $Q$ is isomorphic to $\Gr(2,4)$ via the Pl\"ucker
embedding, $x$ corresponds to a two dimensional vector space $S_x$
in $\C^4$.  Let $\l_x = \P(S_x)$ denote its projectivisation in $\P^3$, then
$$ T_xQ \cap Q = \bigcup_{p \in \l_x} \s_2(p) = \bigcup_{h \supseteq \l_x}
\s_{1,1}(h).$$
The planes from opposite families meet in a line, while those from the
same family meet in the point $x$.  These families of planes embed in $\P^{19}$
as two disjoint conic curves  via the Pl\"ucker embedding of $\Gr(3,6)$,
planes in $\P^5$.

Let 
$$ I_Q = \{ (x, \Lambda) \in X \times \Gr(3,6) \, | \, x \in \Lambda 
\subset T_xQ  \cap Q\}.$$
Then $I_Q$ is isomorphic to the disjoint union of two conic bundles $I_Q^1,
I_Q^2$ on $X$.  

Since 
$$\P(S^\vee_x) \simeq \bigcup_{p \in \l_x} \sigma_2(p)$$
and $$\P(F_x) \simeq \bigcup_{h \supset \l_x} \sigma_{1,1}(h),$$
it follows that these conic bundles are isomorphic to $\P(S^\vee_x)$ and 
$\P(F_x).$
\subsubsection{Singular Quadrics}
When $Q$ is singular, $Q'$ is  a cone over a plane conic and 
hence contains only one family of lines.  So $Q$ contains 
only one family of planes and $I_Q$ embeds in $\P^{19}$ as a conic
bundle on $X$.  The two families of planes in the smooth
quadrics degenerate into the single family of planes in the
singular quadrics.  For $\lambda \in \P^2$, let $Q_\lambda$ 
denote the corresponding
quadric in the net $\cal{Q}$.  Let 
$I \subset X \times \P^2 \times \Gr(3,6)$ be defined by
\begin{equation*} I =\{ (x,\lambda,\Lambda) \, | \, x \in \Lambda \subset 
 T_xQ_\lambda \cap Q_\lambda  \}.
\end{equation*}
Then via the Pl\"ucker embedding of $\Gr(3,6)$ we see that $I$ is a conic
bundle on $X \times M$.  (Recall that $\phi : M \rightarrow \P^2$ is the 
double cover of $\P^2$ branched along $C=V(\det \cal{Q})$.)

In general $I$ is not isomorphic to the projectivisation of a
rank two vector bundle on $X \times M$ 
and this corresponds to the fact that $M$ is a
non-fine moduli space.  However for each $m \in M$, the restriction
$I_{|_{X\times m}}$ does lift to a vector bundle.  We sketch the argument.

A conic bundle lifts to a vector bundle if and only if it has a section.
Fix a line $\l$ in $Q$.  Generically $\l \cap X = \emptyset.$
Then for each $x\in X$ where $x \notin \l$, the line $\l$ meets $T_xQ$
 in exactly one point $p_x$.  There exists a {\it unique plane} $\Lambda_x$
in $T_xQ \cap Q$ {\it in each family}
 such that $\Lambda_x \supset \overline{x p_x},$
where $\overline{x p_x}$ is the line joining the points $x$ and $p_x$.
Note that if $Q$ is singular there is only one such plane $\Lambda$.
Since $T_xQ$ is a holomorphically varying family of hyperplanes this defines
holomorphic sections of $I_Q^1$ and $I_Q^2$ if $Q$ is smooth, and 
a holomorphic section of $I_Q$ if $Q$ is singular.

In terms of Schubert cycles of a smooth quadric we have the following 
description of the section
$$ \l = \s_{2,1}(p_0,h_0) = \{ \l \subset \P^3 \, | \, \l \ni p_0, \l \subset h_0 \}$$
for some $p_0,h_0.$  Since for $x \notin \l$, there is a {\it unique point
$q_x$} such that $\l_x \cap h_0 = q_x$,
so $\overline{p_0q_x}$ is the {\it unique line} contained in $h_0$ that meets
$\l_x$.  Let $p_x$ be the image of this line via the Pl\"ucker embedding
of $\Gr(2,4)$.  Then $\l \cap T_xQ \cap Q=p_x$ and $\s_2(q_x)$ is the unique
plane among the family of Schubert cycles $\s_2$,
containing $p_x$ while $h_x)$ (where $h_x$ is the plane spanned
by $\l_x$ and $q_x$) is the unique Schubert cycle of type $\s_{1,1}$ containing
$p_x$.  As $x$ varies we have maps
\begin{eqnarray*}
s^1: X \rightarrow I^1_Q & &  s^2: X \rightarrow I^2_Q \\
x \mapsto \s_2(q_x) & &  x \mapsto \s_{1,1}(h_x) 
\end{eqnarray*}
giving sections $s^1$ and $s^2$ of $I^1_Q$ and $I^2_Q$ (respectively).


If we restrict attention to a pencil of quadrics in the net then the family of isotropic planes in that pencil corresponds to a double cover of $\P^1$ branched at six points.  The 6 points correspond to the degenerate quadrics in the 
pencil. It is a genus two curve which we denote by $h.$  Let $U$ denote the base locus of the pencil.  Then the same construction as before gives a conic bundle $I_{|_{U \times h}}$ which lifts to a vector bundle and is a universal family \cite{Newstead}.  In particular the restriction $I_{|_{X \times h}}$ lifts to a rank two vector bundle.  It follows that the
Chern invariants of the vector bundle arising from the singular quadric 
are the same as of those coming from the smooth quadrics. 

Next we show that different quadrics give rise to non-isomorphic vector bundles.



\begin{proposition}
Let $Q_1,Q_2 \supset X$ be distinct smooth 
quadrics in $\P^5$ with spinor bundles $S_1$
and $S_2$.  Then $S_{1|_X} \simeq\!\!\!\!\!{/} \,\, S_{2|_X}$.
\end{proposition}
\begin{proof}
By proposition~3.2 it is enough to show that $H^0(S_{1}^\vee \otimes  S_{2|_X})=0$.  
Let $U =Q_1 \cap Q_2$. 


We have the following short exact sequence on $Q_2$,
$$ 0 \rightarrow S_2 \rightarrow \O_{Q_2}^4 \rightarrow F_2 \rightarrow 0.$$
We tensor the above sequence with $F_1(-1)$ and restrict to $U$ to get
$$0 \rightarrow H^0(U,F_1\otimes S_2(-1)_{|_U}) \rightarrow 
H^0(U,F_1(-1)_{|_U}^4).$$

 Corollary \ref {OvanishU}  implies that $H^0(U,F_1(-1)_{|_U})=0$, hence  we get that $H^0(U,F_1\otimes S_2(-1)_{|_U})=0$.

Similarly we have the short exact sequence on $Q_1$,
$$ 0 \rightarrow S_1 \rightarrow \O_{Q_1}^4 \rightarrow F_1 \rightarrow 0.$$
Tensoring this short exact sequence with $S_2(-1)$ we
get 
$$ 0 \rightarrow S_1 \otimes S_2(-1) \rightarrow S_2(-1)^4 \rightarrow 
F_1 \otimes S_2(-1) \rightarrow 0.$$
By Corollary \ref{OvanishU},  $H^0(U,S_2(-1)_{|_U}) = H^1(U,S_2(-1)_{|_U})=0$, so
$$H^1( U,S_1 \otimes S_2(-1)_{|_U}) \simeq H^0(U,F_1 \otimes S_2(-1)_{|_U})=0.$$

Lastly we consider the short exact sequence 
$$0 \rightarrow S_1^\vee \otimes S_2(-2)_{|_U} \rightarrow 
 S_1^\vee \otimes S_{2{|_U}} \rightarrow S_1^\vee \otimes S_{2{|_X}}
\rightarrow 0.$$
 Corollary \ref{OvanishU} shows that ${S_1^\vee}_{|_U} \simeq {S_1(1)}_{|_U}$ hence
the associated long exact sequence in cohomology gives
$$0 \rightarrow H^0(U,S_1 \otimes S_2(-1)_{|_U} \rightarrow 
H^0(U, S_1 \otimes S_2(1) \rightarrow  H^0(X, {S_1^\vee \otimes S_2}_{|_X} 
\rightarrow H^1(U, {S_1 \otimes S_2(-1)}_{|_U}  \cdot \cdot  $$

We have already shown above that 
$ H^1(U,S_1 \otimes S_2(-1)_{|_U}) =0$.
 By the results in \cite{NarasimhanRamanan} 
(and \cite{Newstead}) it follows that $S_{1|_U} \simeq\!\!\!\!\!{/} \,\, S_{2|_U}$, hence $H^0(S_{1}^\vee \otimes  S_{2|_U})= H^0(U,S_1^\vee \otimes S_{2{|_U}}) =0$.  Therefore $H^0(X,S_1^\vee \otimes S_{2{|_X}})=0$.
\end{proof}

We now conclude this Section with a proof of the following theorem.  It is the motivating
example behind our 
more general results in Sections 4 and 5..  See also Example 2.2~\cite{Mukai3}

\begin{theorem}\label{thm:muk}
Let $X$ be the complete intersection of three quadric hypersurfaces $Q_0, Q_1$ and
$Q_2$ in $\P^5$, such that $\Pic(X)$ has rank 1.  let $M$ be the double cover of the plane branched above the sextic curve parameterising the singular quadrics. Then 
$M$ is isomorphic to the moduli space $\M$ of stable rank 2 sheaves with 
$c_1 = \O_X(1)$ and  $c_2=4$.
\end{theorem}

\begin{proof}

Every smooth quadric $Q$ in the net of quadrics generated by $Q_0, Q_1$ and $Q_2$, is isomorphic to $G(2,4)$. The families of two planes in $Q$ correspond to the two distinct classes of Schubert cycles in homology, of dimension 2. These in turn give rise to rank 2 spinor bundles, $S^\vee$ and $F$ with Chern classes,
$c_1 = \O_X(1)$ and  $c_2=4.$  We discuss in Section~3.1.2 how the results in \cite{NarasimhanRamanan} and \cite{Newstead} imply that the family of planes in a singular quadric corresponds to a stable vector bundle on $X$ with the same Chern invariants.  Since $M$ parameterises the families of planes in the net this shows that to each point $m$ in $M$ we can associate a rank 2 vector bundle with $c_1 = \O_X(1)$ and $c_2=4.$  Now Corollary~3.4 shows that in fact these bundles are $\mu$-stable with respect 
to the polarisation.  Hence we get a morphism 
$$ M \rightarrow \M$$
where $\M$ is the moduli space of rank 2, stable sheaves on $X$ with $c_1 = \O_X(1)$ and $c_2=4.$  Proposition~3.5 shows that this morphism is injective.  By Mukai's 
results (Section~2~Theorem~2.16) it follows
that $\M$ is an irreducible K3 surface.  Since $M$ is a compact irreducible K3 surface
which maps injectively to $\M$, it follows that $M \simeq \M$.
\end{proof}



\subsection{Picard groups of higher rank}

The general K3 surface of degree 8 is a complete intersection of three
quadrics, but there are K3 surfaces of degree 8 which are not.  These K3
surfaces must lie on 3 quadrics, but their intersection is 
$\P^1\times \P^2 \subset \P^5$.

\begin{proposition}\label{compint}
Let $X$ be a K3 surface of degree 8 in $\P^5$ with hyperplane  section $H$.
Then the following are equivalent:
\begin{itemize}
\item $X$ is not a complete intersection of three quadrics.
\item $X$ contains an irreducible curve $E$ with $E^2=0$
and $E.H=3$.
\end{itemize}
In this case, the base locus of the net of quadrics that contain 
$X$ is isomorphic to $\P^1 \times \P^2 \hookrightarrow \P^5$, and $X$
is a divisor of bidegree $(2,3)$.  Also $X$ is cut out by a cubic relation
in $\P^5$ in addition to the three quadrics containing $\P^1 \times \P^2.$
The projections $X \rightarrow \P^1$ and $X\rightarrow \P^2$ are given
by the linear systems $|E|$ and $|H-E|$.  Lastly, the net of
quadrics that contain $X$ is spanned by three quadrics of rank 4.  
\end{proposition}
\begin{proof}
If $X$ is a K3 surface of degree 8 in $\P^5,$ then its Hilbert Series 
shows that $X$ is contained in three independent quadrics.  Theorem
7.2 of \cite{SaintDonat} shows that $X$ is the complete intersection of 
these three quadrics unless
\begin{itemize}
\item The generic element of $H$ is hyperelliptic.
\item There exists an irreducible curve $E$ with $E^2=0, E.H=3$
\item $H \equiv 2B+\Gamma$ where $B$ is irreducible of genus 2 
and $\Gamma$ is an irreducible rational curve with $B.\Gamma=1$.
\end{itemize}
If the generic element of $C$ of $|H|$
is hyperelliptic then 
the restriction of $\O_X(H)$ to $C$ is the canonical system
on $C$ which is not an embedding.  So $H$ is not very ample
and we see that the first case is impossible.
For the third possibility,
 suppose that $H=2B+\Gamma$ as above.  For an irreducible
curve $C$ of genus $g$, we have that $C^2 = 2p_a(C)-2 \geq 2g-2$.
So if $H=2B+\Gamma$ then $\H^2=8=4B^2+4+\Gamma^2 \geq 10$.
This contradiction eliminates the last case.
The rest of this proposition follows from Exercise VIII (11) 
\cite{Beauville}.
\end{proof}

We can also choose special quadrics $Q_0,Q_1,Q_3$ so that their
base locus $X,$ contains a conic curve.  In general there is
an 18 dimensional family of such nets of quadrics and the generic 
base locus $X$, has $\rho(X)=2$.  We prove that the geometry
of the associated double cover $M$
has a special feature for this choice of quadrics.

\begin{proposition}\label{tricase}
Let $X$ be the complete intersection of three quadric hypersurfaces $Q_0, Q_1$ and
$Q_2$.  Let $M$ be the double cover of the plane branched above the sextic curve $C$, parameterising the singular quadrics.  
Then $X$ contains a plane conic if and only if the sextic $C$ has a tritangent.
\end{proposition}
\begin{proof}
Let $x_0, \ldots, x_5$ denote homogeneous coordinates on $\P^5$.
  We may assume without loss of generality that $X$ contains the conic $V(x_0^2+x_1^2+x_2^2,x_3,x_4,x_5)$.
Then the quadratic equations that vanish on $X$ are in the ideal 
$(x_0^2+x_1^2+x_2^2,x_3,x_4,x_5)$. Since $X$ is not contained in 
$V(x_3,x_4,x_5) \simeq \P^2$ there exists a basis of $\H^0(2H-X)$ 
of the form
\begin{eqnarray*}
(x_0^2+x_1^2+x_2^2) & +& b_{13}x_3+b_{14}x_4+b_{15}x_5, \\
& & b_{23}x_3+b_{24}x_4+b_{25}x_5, \\
& & b_{33}x_3+b_{34}x_4+b_{35}x_5,
\end{eqnarray*}
where the $b_{ij}$ are linear forms in $x_0,\ldots,x_5.$
Let $y_0,y_1,y_2$ be homogeneous coordinates on the net of quadrics
with respect to the above equations.
So the $6\times 6$
symmetric matrix of linear forms, associated to the net of quadrics,
has the form
$$Q'=\begin{pmatrix}
y_0 & 0   & 0 & \\
0   & y_0 & 0 & A \\
0   & 0   & y_0  \\
   & A^T   &  & D 
\end{pmatrix}$$
where $A,D$ are $3 \times 3$ matrices of linear forms in the 
coordinates $y_0,y_1,y_2$.  The sextic $C=V(\det Q')$ is cut out by the 
determinant of  this matrix which is a polynomial in $y_0,y_1,y_2$.  
To restrict to the line $y_0=0$, we set
$y_0=0$ in the polynomial $\det Q'$ and we get the identity
$$ \det Q'(0,y_1,y_2) =(\det(A)(0,y_1,y_2))^2$$
This is the
square of cubic in $y_1,y_2$ and so the sextic is tangent to the
line $y_0=0$ at 3 points.
\end{proof}

We now come to our main example, namely when $X$ contains a line.  It turns out that
$X$ is isomorphic to $M$.  The proof uses the theory of lattices and 
numeric conditions on $\Pic(X)$ and $\Pic(M)$ \cite{MadonnaNikulin}.


We need the following lemma for our analysis of this case.

\begin{lemma}\label{lem:picM}
Let $X$ be a K3 surface of degree 8 in $\P^5$ that contains a line
and has $\rho(X)=2$.  Then the only classes $C$ in 
$\Pic(X)$ that satisfy $C \cdot C = \pm 2$ are given by;
$$ C = \pm(aH + b\l)\mbox{ where }b-a\sigma=\pm (3+2\sigma)^n$$
for $$n \in \Z
\mbox{ and }\sigma=\frac{1+\sqrt{17}}{2}.$$
\end{lemma}
\begin{proof}
Let $N(x+y\sigma)=x^2+xy-4y^2$ be the norm form on the ring of
integers $\Z[\sigma]$.  We may write $C=aH+b\l$ and the equation $C^2=\pm 2$
is $-2N(b-a\sigma)=\pm 2$.  So we require that $b-a\sigma$ is a unit
in the ring $\Z[\sigma]$.  It is well know that the units are all described as 
powers of the fundamental unit $3+2\sigma=4+\sqrt{17}$, see for example
\cite{quadraticunits} Chapter 11.
\end{proof}

We now give some geometric details about these types of K3 surfaces

\begin{proposition}\label{prop:Hleff}
Let $X$ be a complete intersection K3 surface of degree 8 in $\P^5$ with hyperplane section
$H$.  Suppose that $X$
contains a line $\l$.  The associated degree 2 surface 
$M$ is isomorphic to $X$ and  maps to $\P^2$ via the degree two linear system 
$|h|=|2H-3\l|$
on $X$. 
\begin{enumerate}
\item The ramification sextic curve $C$ is tangent at 24 points to
 a rational octic with 21 nodes.  The inverse image of this octic
is $\l + \l'$ where $\l'$ is a smooth rational curve in $|16H-25\l|$.
The curves $\l$ and $\l'$ meet in 86 points.
\item The sextic $C$ is tangent at 15 points to a quintic with 3 nodes.
The inverse image of this quintic is $D+D'$ where $D \in |H-\l|$
and $D' \in |9H-14\l|$.  The curves $D,D'$ are smooth of genus 3
and meet in 21 points.
So $X$ contains a line if and only if $C$ is tangent to such a quintic.
\item The boundary of the effective cone is generated by the two $(-2)$-curves
$\l$ and $\l'$.  
The closure of the ample cone has boundary generated by $76H-103\l$ and $2H+l$.
\item The involution over $\P^2$ acts on $\Pic(X)$ by $H \mapsto 25H -39\l$
and $\l \mapsto 16H-25\l$.
\end{enumerate}
\end{proposition}
\begin{proof}
The fact that $X \simeq M$ is in \cite{MadonnaNikulin}. 
Consider the curve $\pi(\l)$ in $X$, the image of the line $\l$ under
the mapping $\pi:X \rightarrow \P^2 = |2H-3\l|$.  We see immediately
that $\pi(\l)$ is octic since $(2H-3\l).\l=8$.  Since $\l$
is a smooth rational curve its image must have 21 nodes.
If we let $\l'$ be the image of $\l$ under the involution of $X$ over $\P^2$
we get that $\l'+\l=8h=16H-24\l$.  This proves (1).
A similar analysis on the curve $D \in |H-\l|$ yields the second result.  
Conversely suppose $M$ is a double cover of $\P^2$ branched along a sextic 
$C$ such that $C$ is tangent at 15 points to a quintic with 3 nodes then the inverse
image of the quintic are two genus 3 curves, D and D'.  Working backwards, we see that
$H= 3D-h$ defines a polarisation of degree 8.  The class $\l=2D-h$ is an effective rational 
curve such that $\l.H =1$.  So $\l$ corresponds to the class of a line. Therefore $M$ embeds in $\P^5$ as a degree 8 surface 
containing a line.
 
Next we prove (3).
Proposition VIII (3.8) in \cite{BPV} shows that the effective cone
is generated by effective $(-2)$-curves.  Any $(-2)$-class is either
effective or anti-effective, and if it meets $H$ positively it must be
effective.  Simply plotting the possible $(-2)$-classes, according to 
Lemma \ref{lem:picM} yields the effective cone. 

Lastly, the involution leaves the class $h$ invariant, but switches the curves $D$ and 
$D'$.  It also switches the inverse images of the rational octic tangent at 24 points.  
This implies that $H = D-\l$ goes to $D'-\l' = (9H-14 \l) - (16H - 25 \l) = 25H - 39\l$.
\end{proof}

Note that if $X$ as in proposition~\ref{prop:Hleff} satisfies $\rho(X) = 2$, then the intersection form on $\Pic(X)$ in terms of the generators $H, \l$ is given as 
$\begin{pmatrix}
8 & 1  \\
1  & -2 \\
\end{pmatrix} $.   
In terms of the generators $h, D$ it takes the form
$\begin{pmatrix}

2 & 5  \\
5  & 4 
 
\end{pmatrix}$.  The discriminant of $\Pic(X)$ is $-17$.

\section{Moduli of Rank 2 Sheaves}

We now address the moduli problem on an arbitrary K3 surface of degree 8
in $\P^5$.  It may or may not be a complete intersection.  As before when $X$ is a complete intersection we denote by $M$ the associated double cover of $\P^2$ branched along the sextic curve parameterising degenerate quadrics in the net $\cal{Q}.$  If $X$ is not a complete intersection then as shown before $X$ contains
a curve $E$ such that $E.H = 3$ and $E^2 = 0$.  In this case we let $M = X$ be the double cover given by the degree two linear system $|H-E|$.

 Theorem~\ref{thm:muk} shows that for a generic $X$, i.e. $X$ a complete intersection and $\rho(X)=1$, $M$ is isomorphic to the moduli space of sheaves with Mukai vector 
 $v = (2, H, 2)$.  This is a very special result. In general it is not even possible to check whether a moduli space is
non-empty what to speak of geometric realisations.  In this regard the moduli space $\cM(2,H,2)$ is quite unique and we prove several interesting results.  For instance we show that it is non-empty. We also determine exact conditions for compactness, non-compactness etc.

  Although in general $X$ is a complete intersection, the
double cover $M$ associated to $X$ is not always smooth.  In such situations, $M$ is birational to the moduli space $\M(2,H,2)$ as we prove later.  For
example, if we let $X$ be the desingularisation of a quartic Kummer surface in $\P^3$
then $X$ is a complete intersection of three quadrics in $\P^5.$
The associated double cover $\phi:M \rightarrow \P^2$ is branched along six
lines meeting in 15 points and hence is
singular, \cite{GH}~Chapter 6.  In fact $X$ is isomorphic to the desingularisation of $M$. 
\cite{Beauville} VIII, Problem 9.

We first state some results on linear systems on K3 surfaces which we use below, \cite{Mayer} \cite{SaintDonat}.   
The following proposition is essentially Bertini's Theorem for linear systems on $K3$ surfaces.
\begin{proposition}[Proposition~2.6 \cite{SaintDonat}]
Let $|L|$ be an invertible sheaf on a K3 surface $X$ such that $|L|$ is non empty, $|L|$ has no fixed components and $L^2 > 0$.  Then the generic member of $L$ is an irreducible curve of genus $\frac{1}{2}L^2 + 1$ and $h^1(L) = 0$.
\end{proposition}

\begin{theorem}[Theorem~3.1 \cite{SaintDonat}]
Let $C$ be an irreducible curve on a K3 surface $X$ such that $C^2 > 0$.  Then 
$|C|$ is base-point free.
\end{theorem}

These two results imply that if $|L|$ has no fixed components and $L^2 > 0$ then 
$|L|$ is base-point free.

We also need the following result on hyperelliptic curves in K3 surfaces.

Given a linear system $|L|$ on $X$ with generic element $L$, let $\phi_{L}$ denote the
map given by the line bundle $L$, $\phi_L: X \rightarrow \P^{h^0(L)-1}$.
\begin{theorem}[Theorem~5.6 \cite{SaintDonat}]
Let $B$ be a irreducible curve on $X$ such that ${p_a}(B) = 2$ and let $L = 2B$.  Then 
$\phi_{L}(X)$ is the Veronese surface in $\P^5$.  In fact $\phi_{L} = {v_2}.{\phi_B}$, 
where $v_2 : \P^2 \rightarrow \P^5$, is the Veronese embedding and 
${\phi_B}: X \rightarrow \P^2$ is a 2:1 map, branched over a sextic.
\end{theorem}

We first prove that the Mukai vector $(2,H,2)$ is primitive, i.e. $v$ is not a multiple of another element of $H^*(X, \Z)$.

\begin{theorem}\label{thm:v-primitive}
Let $X$ be a K3 surface of degree 8 in $\P^5$.  Let $H$ denote 
the hyperplane class $\O_X(1)$.  Let $v =(2,H,2)$ in $\tilde{H}^{1,1}(X,\Z).$
Then $v$ is {\rm primitive}.
\end{theorem}
\begin{proof}
The  linear system $|H|$ is base point free since it is very ample.  If not we can replace it with a base-point free, very ample linear system $|H'|$ by removing the fixed components.  The generic element of $|H|$ is a smooth curve of degree 8
in $\P^4.$
Suppose for contradiction that $v$ is not primitive, i.e. $v=(2,2A,2)$ for some element
$A \in N_X$.  Then $|H| = |2A|$ and $A^2=2$.  Since $2A =H$
is very ample and it follows from the definition that $A$ is ample.  This implies that 
$A.C > 0$ for all irreducible curves $C$ in $X$.  Since $A^2 = 2 > 0$, it follows from 
Mumford's vanishing theorem and Serre Duality, that $h^1(-A) = h^1(K_X + A) = h^1(A)  0$. Since $A$ is effective $h^0(-A) = h^2(A) = 0$ so the 
Reimann-Roch theorem implies that $h^0(A) - h^1(A) + h^0(-A) =  2 + A^2/2 = 3$. 

We now show that $|A|$ has no fixed components. The argument is an adaptation of the proof of Alan.~L.~Mayer's result on linear systems of K3 surfaces, Chapter~5~\cite{Friedman}.  Suppose that $A = D_f + D$ where $D_f$ 
is the fixed component.  Then by construction $h^0(D_f) = 1, h^0(A) = h^0(D) = 2 + A^2/2 = 3$.  Then 
\begin{eqnarray*}
D^2  =   A^2  = (D_f + D)(D_f + D) & = & {D_f}^2 + 2{D_f}.D + D^2 \\
\Rightarrow {D_f}^2 + 2{D_f}.D & = & 0 \\
\Rightarrow  {D_f}(D_f + D) + D_f.D & = & 0 \\
\Rightarrow {D_f}.A + {D_f.D} & = & 0 \\
\end{eqnarray*}
Since $A$ and $D$ are effective ample divisors it follows that ${D_f}.A \geq 0$ and 
$ {D_f.D} \geq 0$.  Hence 
\begin{eqnarray*}
{D_f}.A = {D_f}.D & = & 0 \\
\Rightarrow {D_f}({D_f + D}) & = & {D_f}^2 + {D_f}.D = 0 \\
\Rightarrow {D_f}^2 & = & 0
\end{eqnarray*}

If $D_f \neq 0$ then $-{D_f}$ is not effective so $h^0(-{D_f}) = h^2({D_f}) = 0$.  But then 
by Riemann-Roch it follows that $h^0(D_f) \geq 2 + {D_f}^2/2 > 1$ which is a contradiction since we assumed ${D_f}$ is a fixed component and hence satisfies  $h^0({D_f}) = 0$.  This shows that $|A|$ has no fixed components.  It now follows essentially from Bertini's theorem
(Proposition~2.6~\cite{SaintDonat}) that the generic element of  $|A|$  is an irreducible curve of arithmetic genus 2.  Then Theorem~3.1~ 
\cite{SaintDonat} shows that $|A|$ is base point free.  Since $H = 2A$ where  the generic element $A$ of $|A|$ is an irreducible curve with ${p_a}(A) = 2$ it follows from Proposition~5.6~\cite{SaintDonat} that  the generic element $H$ of $|H|$ is hyperelliptic.  Then the restriction of $|H|$ to a generic element $H$ is the canonical system 
on $H$ and is not very ample.  The corresponding map to $\P^5$ is given by 
$\phi_H = v_2 \phi_A$.  The generic element $H$ maps 2:1 onto a line in $\P^2$, followed by the Veronese embedding of degree 4 in $\P^5$. This
 is a contradiction since we assumed $X$ to be a smooth surface of degree $8$ in $\P^5$ with embedding linear system $|H|.$
\end{proof}

We now prove some nice properties of the moduli space  $\cM_H(2,H,2).$ 

\begin{theorem}\label{thm:M-nonempty}
Let $X$ be a K3 surface of degree 8 in $\P^5.$  Let $H$
be a hyperplane section.  
Let $\M_H(2,H,2)$ be the moduli space of stable sheaves (with respect
to $H$) on $X$ with Mukai vector $(2,H,2)$.  Then we have the following: 
\begin{enumerate}
\item The moduli space $\M_H(2,H,2)$ is non-empty.

 \item The moduli space $\cM$ is compact if and only if $X$ does not 
contain an irreducible curve $f$ such that $f^2=0$ and $H.f=4$.

\end{enumerate}
\end{theorem}

\begin{proof}
By Theorem~\ref{thm:v-primitive}, $v = (2,H,2)$ is primitive, so we can apply Theorem
5.4 \cite{Mukai2}, and it follows that $M_H(2,H,2)$ is non-empty.  
This proves part (1).

Now we prove statement (2);
that $\M=\M_H(2,H,2)$ is compact.  According to Proposition
4.1 \cite{Mukai2}, $\M$ is compact if and only if every semistable sheaf $E$
with $v(E) = (2,H,2)$ is stable.  We prove in the subsequent paragraphs that every semistable
sheaf with $v(E) =(2,H,2)$ is stable.

Suppose for contradiction that there exists a sheaf $E$ such that $E$
is semistable but not stable.  Let
$$ 0 = E_0 \subset E_1 \subset \cdots \subset E_n =E$$
be a J.H.S. filtration of $E$ 


Let $F_i = E_i/E_{i-1}$.  Then the $F_i$ are stable and have the same
slope as $E$, so $v(F_i)=(r_i,\l_i,s_i)$ where $r_i/s_i = 2/2=1$
and $\l_i.H/r_i = H.H/2 =4$ by Proposition 2.19 and Remark 2.20 \cite{Mukai2}.
So the only possibility is that there is a rank one subsheaf $E_1$ of $E$
such that we have $$0 \rightarrow E_1 \rightarrow E \rightarrow E/E_1 = F_1
\rightarrow 0.$$
Then $v(F_1) = (1,\l_1,1)$ and $F_1$ is stable and 
$$\mu(F_1) = \mu(E) \Rightarrow H.(\l_1-H/2) =0 \Rightarrow H.\l_1 =4.$$  Since $H$ is ample and $H.(\l_1-H/2)=0,$
the Hodge Index Theorem implies that either $\l_1-H/2 =0$
in $N_X$ or $(\l_1 -H/2)^2<0$.
But $\l_1 -H/2 =0$ would give $H=2\l_1$ which is not possible
by Theorem~\ref{thm:v-primitive}, so $(\l_1-H/2)^2<0$.
Now $v(F_1)^2 = \l_1^2 -2$, and  $(\l_1 -H/2)^2 = \l_1^2 - \l_1.H + (H/2)^2 = \l_1^2 - 2 < 0,$ so we get that 
 $v(F_1)^2<0$.  
Also, $$v(F_1)^2 = - \chi(F_1,F_1^\vee) = - h^0(F_1 \otimes F_1^\vee)
+h^1(F_1 \otimes F_1^\vee)-h^2(F_1 \otimes F_1^\vee).$$
Since $F_1$ is stable we know that it is simple and therefore $h^0(F_1 \otimes F_1^\vee) = h^2(F_1 \otimes F_1^\vee)=1.$ Then $-2 \leq v(F_1)^2 < 0.$  Since the intersection pairing on $\H^2(X, \Z)$ is even for a K3 surface, 
we get $v(F_1)^2 = -2,$ which implies $\l_1^2=0.$
Now Riemann-Roch and Serre duality give
$$ \chi(\l_1) = h^0(\l_1) - h^1(\l_1) + h^0(-\l_1) = 2$$ which implies that either $\l_1$ is effective or $-\l_1
$ is effective.  Since $H$ is ample and $H.\l_1 >0$
we see that $\l_1$ is effective.
So $| \l_1 | =|kf|$ for some $f \in N_X$ such that $f^2=0$
and $p_a(f)=1$.  Then $H. \l_1= k.H.f=4$.  Since the generic
element of $|H|$ is non-hyperelliptic it follows that 
$H.f \geq 3$ for any $f$ such that $p_a(f) =1$ by \cite{SaintDonat} Section 7, Remark~7.1.
So $k=1$ and the generic element of $| \l_1 |$ is an irreducible
curve with genus 1.

Since $v(F_1)^2 = -2$ and $F_1$ is a rank $1$ torsion free sheaf, we will
show that 
$F_1$ is locally free. Let $c_1, c_2$ denote the Chern invariants of $F_1.$ Then we have
\begin{eqnarray*}
v(F_1)&  = & (1, c_1, c_1^2-c_2+ 1)\\
& = & (1,\l_1,0-c_2+1) \\
 & = & (1, \l_1,1). 
\end{eqnarray*}
So $c_2(F_1)=0$
If $F_1 = \O(\l_1) \otimes I_Z$ for some zero dimensional
subscheme $Z,$ then $c_2(F_1) = \l_1^2 +\length(Z).$
which implies $\length(Z)=0$ and that $F_1$ is a line bundle.

So if $F_1$ exists, then we have the extension 
\begin{equation}\label{extension}
0 \rightarrow E_1 \rightarrow E \rightarrow F_1 \rightarrow 0.\end{equation}
Any torsion free rank 2 sheaf $E$ on a surface is of the form
$$ 0 \rightarrow L_1 \otimes I_{Z_1} \rightarrow E \rightarrow 
L_2 \otimes I_{Z_2} \rightarrow 0$$
where $L_1,L_2$ are line bundles and $Z_1$ and $Z_2$ are zero dimensional
schemes, \cite{HuybrechtsLehn} Equation 5.1, Section 5.1, p.123.  We know that $F_1,$ in the extension (\ref{extension}) above, is a line bundle so $E_1 = L \otimes I_{Z}$ for some $L = \O(\sigma)$
with $\sigma \in N_X$ and $Z$ a zero dimensional subscheme.
Now $c_1(E_1 ) = c_1(L) = \sigma$
and $c_2(E_1) = \length(Z)$.
So 
\begin{eqnarray*}
c_1(E) & = & c_1(F_1)+c_1(E_1) \\
\Rightarrow H &=& \l_1 + \sigma \\
\mbox{and} & \ \ & \\
c_2(E) &=& c_1(F_1) \cdot c_1(E_1) + \length(Z)\\
\Rightarrow 4 & = & \l_1 \cdot \sigma + \length(Z)
\end{eqnarray*}

Since
$\H^2 = (\l_1 + \sigma)^2 = 8, {\l_1}^2 = 0$ and 
$H \cdot \l_1 = 4$ we get $\l_1 \cdot \sigma = \l_1\cdot \sigma + \sigma^2 = 4.$
So $\sigma^2=0$.
Now $c_2(E) = 4 = \sigma \cdot \l_1 + \length(Z)$ and  $\sigma \cdot \l_1= 4$ so $\length(Z)=0$.
This implies that $E_1 =\O(\sigma)$ is a line bundle and so $E$
is a vector bundle that is an extension of $\O(\l_1)$ by 
$\O(\sigma)$.  



Now if $N_X$ contains classes $\sigma$ and $\l_1$ satisfying $\sigma^2 = \l_1^2 = 0$ and $\sigma \cdot \l_1 = 4$ then there exists and extension $E= \O(\sigma) \oplus \O(\l_1)$ by \cite{Friedman}~Chapter 4~ Proposition~21~(ii).  Then $E$ is strictly semistable, i.e. it is semistable but not stable.

Since we assume in the hypotheses that $N_X$ does not contain such classes it follows that $E$ is stable so $\M$ is compact.  On the other hand if 
$N_X$ does contain such classes then $E \simeq \O(\sigma) \oplus \O(\l_1)$
is semistable so $\cM$ is not compact, which proves the necessity 
in statement (2).
\end{proof}

The following theorem gives a geometric realisation of $\M_H(2,H,2)$ for an arbitrary degree 8 K3 surface $X$.
\begin{theorem}\label{thm:M-bir}
Let $X$ be a complete intersection of quadrics in $\P^5$ and $M$ the double cover 
of the plane branched along a sextic parameterising the degenerate quadrics.  Let 
$\M$ be the moduli space of sheaves on $X$ with Mukai vector $(2,H,2)$. Then,
\begin{enumerate}
\item $\M$ is
birational to $M.$  

\item In addition if $X$ does not contain an
irreducible curve $f$ such that $f^2=0$ and $f.H=4$, then 
$\M$ is the minimal resolution of $M$.
\end{enumerate}
\end{theorem}

\begin{proof}
The surface $X$ lies on 3 independent 
quadrics $Q_0,Q_1,Q_2$.  Let $\cal{Q}$ be the net spanned by them and let
$V(\det \cal{Q})$ be the plane sextic curve parameterising the degenerate quadrics.
Let $M \rightarrow \P^2$ is the associated double cover of $\P^2$
branched along $V(\det \cal{Q})$.  Each smooth quadric $Q$
in the net corresponds to two vector bundles on $X$ with Mukai vector 
$(2,H,2)$.  The smooth points of the plane sextic $V(\det(\cal Q))$ correspond to quadrics which have a node.  As seen in Section~3, each of these quadrics corresponds to a single vector bundle with Mukai vector $(2,H,2)$.  By Proposition~3.5 distinct quadrics give rise to non-isomorphic vector bundles so we get an injective morphism from an open subset of 
$M$ to the moduli space  $\M$. 

We first show that $\M$ is birational to $M$.  
We will use an argument similar to that of p.391 of \cite{Mukai2}. 
To do so we have to consider the natural compactification $\overline{\M}$ of $\M$ given by considering all {\it semistable sheaves} on $X$ with given Mukai vector.
Let $\boldsymbol{X} \rightarrow F_8$ denote the moduli space of polarised K3 surfaces $(X, H)$ of degree $8$. By the Torelli theorem on $K3$ surfaces,  it is irreducible.  It contains a Zariski open subset $U$, such that every fibre of 
 $\boldsymbol{X} \rightarrow U$ is a complete intersection. Let $H_s$ denote the polarisation on $X_s$.  Correspondingly 
 we get another family $\boldsymbol{M} \rightarrow U$ of polarised K3 surfaces of degree two such that for each fibre $\boldsymbol{X}_{|_s} = X_s$,  the corresponding fibre $\boldsymbol{M}_{|_s} = M_s$ is the K3 surfaces associated to $X_s$.  
 
 By Maruyama's results the coarse (relative) moduli space 
 $\pi: \overline{\boldsymbol{\mathcal{M}}} \rightarrow U$ of {\it semistable sheaves} 
 $\boldsymbol{X} \rightarrow U$ and each fibre of $\pi$ is canonically isomorphic to 
 the moduli space $\overline{\M}_H(2,H,2)$ of {\it semistable} sheaves on the corresponding fibre of 
 $\boldsymbol{X} \rightarrow U$ \cite{Maruyama}.   By Langton's result \cite{Langton} the family 
 $\overline{\boldsymbol{\mathcal{M}}} = \{ \overline{\M}_{s_{H_s}}(2,H_s,2) \}_{(X_s, H_s) \in \boldsymbol{X}}$ is proper over $U$, and as a consequence over each fibre, the moduli space $\overline{\M_s}$ of semistable sheaves on $X_s$ is compact.
 
 There is a Zariski dense open subset $V$ of $U$ over which $\rho(X_s) = 1$.  Then by 
 Theorem~\ref{thm:muk}, we have that 
 $\overline{\M}_s \simeq {\M}_s \simeq M_s$ for each $s \in S$.  In other words the moduli space 
 $\M_s$ is a compact {\it irreducible} K3 surface which is isomorphic to $M_s$.  Therefore 
 there is an isomorphism $\overline{\boldsymbol{\M}}_{|_V} \rightarrow 
{\boldsymbol{M}}_{|_{V}}$.  By Corollary~3.18~\cite{Mukai2}, the complement of 
 $\M_s$ in $\overline{{\M}_s}$ is a discrete set.  It follows that the family 
 $\overline{\boldsymbol{\M}}_{|_U}$ has only one component, so in fact 
 for all $\ s \in U$ we have that $\M_s$ is a smooth symplectic manifold with only one component.  Since we have an injective morphism from an open subset of $M_s$ to $\M_s$, it follows that $\M_s$ is birational to $M_s$.

 The proof of the second statement now follows easily. 
 By Theorem~\ref{thm:M-nonempty} we see that $\M$ is non-empty and compact.  By Mukai's Theorem~\ref{thm:mukai} we see that it is a an irreducible K3 surface.  We know $\M$ is birational to $M$ so $\M$ is a minimal resolution of $M.$
\end{proof}

In fact as a result we get the following Proposition

\begin{proposition}\label{prop:M-noncompact}
There is an $18$ dimensional family of K3 surfaces of degree $8$ in $\P^5$ such that 
for a generic element $X$, the moduli space  $\M_H(2,H,2)$ is an open subset of the resolution of $M$, but is not
compact.  In addition, $\M_H(2,H,2)$ is a non-fine moduli space.
\end{proposition}

\begin{proof}
Consider the rank two lattice $N \simeq \Z \oplus \Z$ with generators $f_1, f_2$ and 
pairing ${f_1}^2 = {f_2}^2 = 0, f_1 \cdot f_2 = 4.$  Then $N$ is an even lattice with signature $(1,-1)$ and 
$N^\vee / N \simeq \Z/ 2\Z \oplus \Z/2\Z.$  
By Theorem~1.14.4~\cite{Nikulin} there exists a unique primitive embedding 
of $N$ into the K3 lattice $L.$  Then $N^\perp \cap \Omega$ is an $18$ 
dimensional subset of $\Omega$ and by the surjectivity of the period 
mapping (Theorem~\ref{thm:k3period}) each point of $N^\perp \cap \Omega$ 
occurs as the period point of a marked K3 surface $X.$  The generic such 
surface $X$ has $\Pic(X) \simeq N.$
Since there are no classes with 
square -2, we see that the class $H:= f_1 + f_2$ is base point free and has 
no fixed component (Proposition 2.6 and Theorem 3.1~\cite{SaintDonat}, also stated in Proposition 4.2 and Theorem 4.1~Section 4).  Hence $H$ is ample and defines an embedding of $X$ in 
$\P^5$ as a smooth degree $8$ surface.  The moduli space $\M = \M_H(2,H,2)$ is nonempty 
as seen already in Theorem~\ref{thm:M-nonempty}.  However 
$E = \O(f_1) \oplus \O(f_2)$ is semistable but not stable with 
$v(E) = (2,H,2),$  so $\M$ is not compact.

By Proposition~\ref{compint} $X$ is not a complete intersection if there exists an 
irreducible curve $E$ such that $E^2=0$ and $E.H = 3$.  Suppose for contradiction that 
such a curve $E$ exists.  Then $E= af_1 + bf_2, a,b \in \Q$ and $2ab = 0, 4a +4b = 0$.  The only possibilities are $a=0, b = 3/4$ and $a=3/4, b = 0$.  Suppose $a=0, b = 3/4$, 
then $E=3/4f_2$, but $f_2$ is an irreducible curve of genus one and $E$ is a rational 
multiple of it, so $E$ cannot be an irreducible curve in $X$.  
This shows that $X$ is a complete intersection.  Then by Theorem~\ref{thm:M-bir}, it follows that $\M$ is birational to $M$.

To show that in general $\M$ is non-fine we note that the generic element $X$ in this family has $\rho(X) = 2.$  Then $\sigma_{min} = 2$ and hence $\M$ is non-fine (\cite{Mukai2}~Theorem~A.5 and Remark~A.7). 
\end{proof}

Returning to the case when $N_X$ does not contain a sublattice 
with basis $f_1,f_2$ where $f_1^2=f_2^2=0$ and $f_1.f_2=4$, we prove
a much stronger result, namely that every element of $\M_H(2,H,2)$
is $\mu$-stable with respect to $H$ and locally free.

\begin{theorem}\label{thm:M=Mslope}  
Let $X$ be a K3 surface of degree 8 in $\P^5$ and suppose that
$X$ does not contain an irreducible curve $f$ where $f^2=0$ and $f.H=4$, and let 
$\M=\M_H(2,H,2)$ be as in Theorem~\ref{thm:M-nonempty}.
Then every element of $\M$ is
$\mu$-stable with respect to $H$ and is locally free.
\end{theorem}
\begin{proof}
If $\cM$ is compact then every semistable sheaf $E$ in $M$ is stable.  Theorem~\ref{thm:M-nonempty} proves that $\M$ is compact.  Since $E$ is stable it is 
$\mu$-semistable with respect to $H$.
Assume for contradiction that $E$ is not $\mu$-stable.  It is clear that we are considering $\mu$-stability with respect to $H$ so from now on by $\mu$-stability we mean $\mu$-stability with respect to $H$.  Then $E$ has a proper rank $1$ quotient sheaf $F_1$ with $\mu(F_1) = \mu(E) = 4.$  Let $v(F_1) = (1, c_1, s_2).$  Recall that $s_2 = {c_1^2}/2 - c_2 + 1.$  Since
$\mu (F_1) = \mu (E)$ we have $H.(c_1 - H/2) = 0.$  Therefore
\begin{eqnarray*}
v(F_1)^2 &=& (c_1 -H/2) +H/2)^2 - 2s_2\\
               &=& (c_1-H/2)^2 + (H/2)^2 - 2s_2\\
               &=& (c_1-H/2)^2 + 2(1-s_2)
\end{eqnarray*}               
Now $v=(2,H,2)$ is primitive so $c_1 - H/2$ is not equal to zero.  
However $H$ is ample and $H. (c_1 - H/2) = 0,$ so
 by the Hodge index theorem $(c_1 - H/2)^2  < 0.$  On the other hand since $E$ is stable, the normalised Hilbert polynomial 
$$p_{H,E}(n) = \frac{\chi( E \otimes H^n)}{\rank(E)} <  p_{H,F_1}(n) = \frac{\chi(F_1 \otimes H^n)}{\rank(F_1)}$$ for large $n.$  



 

It is easy to calculate these polynomials using the 
usual Hirzebruch-Riemann-Roch Theorem.  We take a 
resolution of $E$ to calculate $\chi(E \otimes H^n)$.
$$ P_{H,E}(n) = 2 +4n +4n^2 < P_{H,F_1}(n) = 1+s_2 +n (c_1.H) +4n^2.$$
Since $\mu(F_1) = c_1.H = \mu(E) =4.$  It follows that $2< 1+s_2,$ so $1-s_2<0$.  This 
inequality along with $(c_1 - H/2)^2 < 0$ implies that 
$$v(F_1)^2 = (c_1 - H/2)^2 + 2(1 - s_2) < -2.$$
On the other hand $F_1$ is stable implies $F_1$ is simple so 
$h^0(F_1 \otimes {F_1}^\vee) = 1$.  This gives 
$$v(F_1)^2 = \sum_i (-1)^{i+1} \dim \Ext^i(F_1,F_1) = -2 h^0(F_1 \otimes {F_1}^\vee) + \dim \Ext^1 (F_1, F_1)  \geq -2$$ and we get a contradiction.
So every such $E$ in $\cM$ is also $\mu$-stable.

Next we prove that every $E$ is also locally free.
If $E$ is $\mu$-stable then $E^{\vee\vee}$
is also $\mu$-stable.
Consider the short exact sequence
$$ 0 \rightarrow E \rightarrow E^{\vee\vee} \rightarrow E^{\vee\vee}/E \rightarrow 0.$$
where the support of $E^{\vee\vee}/E$ is some zero dimensional subscheme $Z$ of $X$.
Then $v(E^{\vee\vee})=(2,H,2-\length(Z))$
since $c_2(E^{\vee\vee}) = c_2(E) -\length(Z)$.  
So $E \mapsto E^{\vee\vee}$ defines a morphism 
$\M \rightarrow M(2,H,2-\length(Z))$ 
and $\cM(2,H,2-\length(Z))$ is non-empty.
However 
$$\dim \cM(2,H,2-\length(Z)) = 2+(8-4(2+\length(Z))) = -4\length(Z)+2.$$
So if $\length(Z) > 0$, then $\cM(2,H,2-\length(Z))$ has negative dimension
which is a contradiction.
Therefore $E \simeq E^{\vee\vee},$ and so $E$ is locally free.
\end{proof}

We now define an equivalence relation on 
Mukai vectors.
\begin{definition}\label{def:twist}
Let $X$ be a K3.  Let $w=(w^0,w^1,w^2),v=(v^0,v^1,v^2) \in \tilde{H}^{1,1}
(X,\Z)$.  We say that $w$ is {\it equivalent} to $v$, $w \sim v$
if there exists a line bundle $L$ such that $w = \ch(L)\cdot v$.
So $$w = (v^0,v^1+v^0c_1,v^2+v^1.c_1+v^0.c_1^2/2)$$ where $c_1=c_1(L)$.
\end{definition}
The next theorem follows easily since tensoring by line bundles
preserves slope stability.
\begin{theorem}\label{thm:M(w)}
Let $X$ be a K3 surface of degree 8 which does not contain an irreducible curve
$f$ such that $f^2=0,$ and $f.H=4$.  Let $w$ be in $\tilde{H}^{1,1}(X,\Z)$
such that $w \sim (2,H,2)$.  Then $\M_H(w)$ is isomorphic to $\M_H(2,H,2)$.
\end{theorem}
\begin{proof}
Since $w \sim (2,H,2),$ there exist a line bundle $L$ such that $w = \ch(L)\cdot (2,H,2).$  
Tensoring by a line bundle preserves $\mu$-stability so we get a morphism
\begin{eqnarray*}
\M_H(2,H,2) &\rightarrow & \M_H(w) \\
E & \rightarrow & E \otimes L.
\end{eqnarray*}
Since $\M_H(2,H,2)$
is non-empty, compact it is an irreducible K3 surface (Theorem~\ref{thm:M-nonempty}).  Moreover since every element of $\M_H(2,H,2)$ is $\mu$-stable and locally free(Theorem ~\ref{thm:M=Mslope}) , it follows that the image of 
$\M_H(2,H,2)$ is a compact connected component of $\M_H(w).$  Then  Proposition~4.4 \cite{Mukai2} implies that $\M_H(w)$ is compact and 
irreducible hence isomorphic to  $\M_H(2,H,2).$ 

\end{proof}

\begin{lemma}
Let $X$ be a K3 and $v =(v^0,v^1,v^2) \in \tilde{H}^{1,1}(X,\Z)$,
be isotropic with $v^0 \neq 0$.  Then $v^2=(v^1)^2/2v^0$ is determined by $v^0$ and $v^1.$  Also $v^1$ is determined modulo $v^0 \Pic X,$ up to equivalence.
\end{lemma}
\begin{proof}
Since $v$ is isotropic, we have that $2v^0v^2=(v^1)^2$.  
When we twist by a line bundle $L$ with $c_1=c_1(L),$ we get 
$$w =(v^0,v^1+v^0c_1,v^2+v^1.c_1+v^0.c_1^2/2).$$
so $w^1=v^1+v^0c_1$ and $w$ is isotropic.\end{proof}

Finally we prove a lemma showing that if $\rho(X)=2$, and $X$ contains a line, then 
in fact $X$ is a complete intersection.  

\begin{lemma}\label{lem:X-ci}
Let $X$ be a smooth degree $8$ K3 surface in $\P^5$ with hyperplane class
$H$ and suppose that $X$ contains 
a line $\l$ and $\Pic(X)\otimes \Q = \Q H \oplus \Q l.$
  Then $X$ is a complete intersection of three independent quadrics 
$Q_0$, $Q_1$, $Q_2.$ 
\end{lemma}
\begin{proof}
By Proposition~\ref{compint}, we need to show that $X$
does not contain an irreducible curve $E$ such that 
$E^2=0$ and $E.H=3$.  Let $E=aH+bl$ we see that $E.H=3=8a+b$.
and $E^2=0=8a^2+2ab-2b^2$ has no non-trivial rational solutions for $a,b$.
\end{proof}

We apply the above results to our calculation of the Fourier-Mukai Transform in the next section.

\section{Cohomological Fourier-Mukai transform}\label{sec:FM}

Let $X_s$ be a K3 surface with ample class $A_s$ and $\M_s = \M_{A_s}(v_s)$ a K3
 surface which is a moduli space of sheaves on $X_s$ with primitive, isotropic Mukai vector $v_s.$  Suppose $\M_s$ is not fine with similitude $\sigma.$  Let $\cE_s$ be a quasi-universal sheaf on $X_s \times M_s.$  Then there exist a marked {\it degeneration} $X_0$ of $X_s$ with polarisation $A_0$ such that the corresponding moduli space $\M_0 = \M_{A_0}(v_0)$ is a fine moduli space of sheaves on $X_0$ with universal sheaf ${\cE}_0.$  
This means that there is a marked family $\boldsymbol{X} \rightarrow S$ (where $S$ is an open neighbourhood of $0 \in \C$ with $s \in S$) of polarised K3 surfaces 
and a corresponding marked family   $\boldsymbol{\mathcal{M}} \rightarrow S$ of polarised K3 surfaces such that, for $s \in S,$
\begin{enumerate}
\item $\boldsymbol{X}_{|_s} = X_s$ is a K3 surface with polarisation $A_s$ and 
$\boldsymbol{\mathcal{M}}_s = \M_s = \M_{A_s}(v)$ is the moduli space of sheaves on $X_s.$  

\item A flat family of sheaves $\boldsymbol{\mathcal{E}}$ exists on $\boldsymbol{X} \times \boldsymbol{\mathcal{M}}$ such that its restriction
$\boldsymbol{\mathcal{E}}_{|_{X_s \times \M_s}} = {\cE_s}_{|_{X_s \times \M_s}}$ is the quasi-universal sheaf 
of the moduli problem on $X_s.$  

\item At $s=0$ the moduli problem is fine and 
$$\boldsymbol{\mathcal{E}}_{|_{X_0 \times \M_0}} \simeq {\cE}_0 \otimes \pi^*_{\M_0}{W}$$
where $W$ is a rank $k\sigma$ vector bundle on $\M_0,$ for some positive
integer $k$.

\item  We have an identification
$\H^* (X_s \times \M_s) \simeq \H^*(X_0 \times \M_0)$ and
$\ch (\boldsymbol{\mathcal{E}}) \in {\H}^*(\boldsymbol{X} \times \boldsymbol{\mathcal{M}}, \Q)$ is constant.

\end{enumerate}

In particular the isomorphism in $(4)$ implies that 
$$\ch {\cE}_s = \ch {\cE_0} \cdot \pi_{\M_0}^*{\ch W} \in H^*(X_0 \times \M_0, \Q).$$

The $\H^0$-component of $f_\cE(x)$ is given by $(x,v).$  Since
$$f_{\cE \otimes \pi_M^* W}(x) = f_\cE(x) \cdot \frac{\ch W^\vee}{\rank(W)}$$
the $\H^2$-component of $f_\cE (x)$ for $x \in v^\perp$ is independent of the choice of a quasi-universal sheaf.  

We prove this below.
\begin{proposition}\label{prop:fm1}
Let $\M_0 = {\M_0}_{A_0}(v)$ be a fine moduli space of sheaves on $X_0.$  Let $X_s$ be a marked deformation of $X_s$ and 
$\M_1 = {\M}_{A_1}(v)$ the corresponding moduli space of sheaves on $X_1.$  Assume that $\M_1$ is not fine.  Let $\cE_1$ be a quasi-universal sheaf with similitude $\sigma$ on 
$X_1 \times \M_1$ and $\cE_0$ a universal sheaf on $X_0 \times \M_0.$  
Let $x \in v^\perp / \Z v, f_{\cE_1}(x) = (a^0,a^1,a^2)$ and 
$f_{\cE_0}(x) = (b^0,b^1,b^2).$  Then $a^0 = b^0 = 0,$ and  $a^1 = b^1 \in \H^*(X_0 \times \M_0, \Q)$ 
\end{proposition}

\begin{proof}
Let $\pi_{X_0}, \pi_{\M_0}$ denote projections onto the first and second factor respectively.  

Then 
$\cE_1$ is a flat deformation of $\cE_0 \otimes \pi_M^* W$ for some vector bundle $W$ of rank $\sigma$.
The marking gives isomorphisms 
$\H^*(X_1, \Z) \simeq \H^*(X_0, \Z)$, $\H^*(M, \Z) \simeq \H^*(M_0, \Z)$ and
$$\H^*(X_1 \times \M_1, \Z) \stackrel{\psi}{\rightarrow} \H^*(X_0 \times \M_0, \Z).$$

Via the isomorphism $\psi$ we get $\ch {\cE_1} = {\ch \cE_0} \cdot  {\ch \pi_{\M_0}^*W}$  and 
$$Z_{\cE_1} = \pi^*_{X_0} \sqrt{\td_{X_0}} \cdot \ch({\cE_1}^\vee) \cdot \pi^*_{\M_0}\sqrt{\td_{\M_0}}/\sigma = Z_{\cE_0} \cdot \frac{\ch (\pi_{\M_0}^*W^\vee)}{\sigma}.$$   
We write 
$f_{\cE_1}(x) = b^0 + b^1.t + b^2.t^2$ and $f_{\cE_0}(x) = a^0 + a^1.t + a^2.t^2$  where 
$a^i, b^i \in \H^{2i}(M, \Q).$ 
Recall that 
$f_{\cE_1}(x) = \pi_{{\M_1}*}(Z_{\cE_1} \cdot \pi_{X_1}^*(x)).$  So we get
\begin{eqnarray*}
f_{\cE_1}(x) & = & \pi_{{X_1}*} (Z_{\cE_1} \cdot \pi_{X_1}^*(x))\\
&=& \pi_{{\M_0}*}\left(Z_{\cE_0} \pi_{X_0}^*(x)\cdot \frac{\ch (\pi_{\M_0}^*W^\vee)}{\sigma} \right)\\
&=& f_{\cE_0}(x) \cdot \frac{\ch W^\vee}{\sigma}\end{eqnarray*}
by the projection formula.
So we see that 
\begin{eqnarray*}
b^0 + b^1.t + b^2.t^2 & = & (a^0 + a^1.t + a^2.t^2)(1 + \frac{c_1(W)}{\sigma}.t + (\ldots).t^2) \\
b^0 + b^1.t + b^2.t^2 & = & a^0 + \left(a^1 + a^0.{\frac{c_1(W)}{\sigma}}\right).t  + (\ldots)t^2
\end{eqnarray*}

By Lemma~4.11~\cite{Mukai2}, we have that $f_\cE(v) = (0,0,1)$ is the fundamental class and by the remark preceding Lemma~4.11~\cite{Mukai2}, for any $x \in \H^*(X, \Z)$ the $\H^0$-component of $f_\cE (x)$ is equal to $(x,v).$  So $f_\cE (x) \simeq (x,v) \pmod{t}.$
So for $x \in v^\perp$ such that $x \neq \Z v$ we get $a^0 = b^0 = 0$ and 
$b^1 = a^1.$ 
\end{proof}

\begin{lemma}\label{lem:fm2}
Let $X$ be a generic K3 surface of degree $8$ in $\P^5$ and let $v=(2,H,2)$.
Then $\rho(X) = 1$  
and $\M(v) \simeq  M,$ where $\phi:M \rightarrow \P^2$ is the double cover associated to $X.$  Let $h = \phi^{*}(\O_{\P^2}(1))$ be the polarisation on $M.$  Then the image of
the coset $(1,0,-1) + \Z v$ in $v^\perp / \Z v$ has two possibilities, namely,
$$f_\cE : (1,0,-1) + \Z v \rightarrow (0,\pm h,0).$$
\end{lemma}
\begin{proof}
By Mukai's theorem $$f_\cE : v^\perp/ \Z v \cap \tilde{\H}^{1,1}(X, \Z) \rightarrow \Pic(M)$$ is an isomorphism.  Also $f_\cE (v) = (0,0,1)$ the fundamental class of $M$ by Lemma~4.11~\cite{Mukai2}.  In this case $$v^\perp/ \Z v \cap \tilde{\H}^{1,1}(X,\Z)$$ is generated by 
the coset $(1,0,-1) + \Z v.$  Since $((1,0,-1) + \Z v)^2 = 2$ its image in $\tilde{\H}^{1,1}(M, \Z)$ has to be an an element that squares to $2.$  So 
$$f_\cE : (1,0,-1)+ t\cdot v \rightarrow (0,\pm h,0).$$
\end{proof}

When $X,$ a K3 surface of degree $8$ in $\P^5,$ contains a line $\l$ and has $\rho (X) = 2$  it turns out that $X$ has another special feature.  
The linear system $|H-\l|$ embeds $X$ in $\P^3$ as a quartic surface
containing a twisted cubic curve, the image of $\l$.  Then taking the residual intersection of 
the image of $X$ in $\P^3$
with quadric surfaces containing the twisted cubic
realises $X$ as a double cover of $\P^2$.  Since $X$ is a K3 surface, 
it has to be
branched along a sextic.  The polarisation of degree $2$ on $X$ which is the pull back of $\O_{\P^2}(1)$ corresponds to the class 
$h:= 2H- 3\l.$  In fact $X \simeq M$ where $M \rightarrow \P^2$ is the associated double cover of $\P^2.$  The proof that $X \simeq M$ involves arguments using theory of lattices and is a special case of the results in \cite{MadonnaNikulin}.  The isomorphism classes of all such K3 surfaces is a Zariski open subset of an $18$ dimensional family $\boldsymbol{Y}$ contained in the $19$ dimensional family of all K3 surfaces of degree $8$ in $\P^5.$

We make a small change of notation and let $X_0$ now denote a smooth K3 surface in $\P^5$ of degree 8 which contains a line $\l$ and has $\rho(X_0) = 2.$    
We use this notation to make the deformation theory argument clear.
Recall that $X_0 \simeq \M_0 \simeq M_0$ where $\phi:M_0 \rightarrow \P^2$ is the double cover of the plane associated to $X_0.$  The degree $2$ polarisation on $X_0$ is given by the class $h:=2H - 3\l.$  Let $\boldsymbol{X} \rightarrow S$ be a marked deformation of $X_0$ transverse to the family $\boldsymbol{Y}.$  Then the generic element $X_s = \boldsymbol{X}_{|_s}$ is a K3 surface with $\rho(X_s) = 1.$  The corresponding moduli space $\M_s = \M_H(2,H,2)$ is non-fine.  Let $\cE_s$ be the quasi-universal sheaf on $X_s$ and $\cE_0$ the universal sheaf on $X_0 \times \M_0.$  Let
$$f_{\E_0} : \tilde{\H}(X_0, \Z) \rightarrow \tilde{\H}(M_0, \Z)$$
 be the isomorphism of lattices given by the Fourier-Mukai map on cohomology 
as in Theorem~\ref{thm:mukai-map}.  Then we compute 
$$f_{\cE_0}: \tilde{\H}^{1,1}(X_0, \Z) \rightarrow \tilde{\H}^{1,1}(M_0, \Z)$$ explicitly up to twists by line bundles and addition by some constant factors.  



\begin{theorem}\label{thm:mukai-map}
Let $X$ be a K3 surface of degree 8 in $\P^5$ that contains a line
and has $\rho(X)=2$, and let $\M= \M_H(2,H,2)$.
Recall that 
$X \simeq \M \simeq M .$ Then 
\begin{enumerate}
\item The cosets $x = (1,0,-1) + \Z v$ and $w = (1, -H +2\l, -4) + \Z v$ form a basis for $v^\perp/ \Z v \cap \tilde{\H}^{1,1}(X_0, \Z).$
\item The lattice generated by $x$ and $w$ has the intersection pairing $x^2 = 2, x\cdot w = 5, w^2 = 4.$  It is mapped isometrically to $\Pic (M)$ via the Fourier-Mukai map on cohomology induced by 
$$f_\E: v^\perp / \Z v \cap \tilde{\H}^{1,1}(X, \Z)  \rightarrow  {\H}^{1,1}(M, \Z).$$ 
\item $f_\E (0,0,1) = (2,H,2)$ up to twists by line bundles. 
\end{enumerate}
\end{theorem}
\begin{proof}

Let $v = (2,H,2).$  Then $f_\E(v) = (0,0,1)$ the fundamental class by \cite{Mukai2}~Lemma~4.11.  Since $f_\E$ is an isometry, 
$f_\E$ maps $x + \Z v$ to an element in $\Pic (M)$ that squares to $2.$  
Such elements in $\Pic(M)$ are characterised by 
Lemma~\ref{lem:picM} and include $2H - 3\l$ and $2H + 5\l$.  To determine
which  element, we 
consider now a marked deformation $\boldsymbol{X} \rightarrow S$ of $X$ such that $X = \boldsymbol{X}_{|_{s=0}}$ and the generic element $X_s = \mathcal{X}_{|_s}$ has $\rho(X_s) = 1.$  Here $S$ is an open disk as before.  Then $\M_s$ is in general a non-fine moduli space of sheaves on $X_s.$   The marking gives isomorphisms $\H^*(X,\Z) \simeq \H^*(X_s, \Z)$ and $\H^*(M, \Z) \simeq \H^*(M_s, \Z).$ By 
Proposition~\ref{prop:fm1} the $\H^2$-component of 
$f_\cE (x)$ is equal to the $\H^2$-component of $f_{\cE_s} (x).$  By Lemma~\ref{lem:fm2} $f_{\cE_s}(x + \Z\cdot v) = (0, \pm h, 0).$  Let us assume that 
$f_{\cE_s}(x + \Z\cdot v) = (0, h,0)$.  The case 
$f_{\cE_s}(x + \Z\cdot v) = (0,-h,0)$ is similar.  So 
$f_\E(x + \Z\cdot v) = (0,h,0) = (0, 2H - 3\l, 0).$   
A basis for $\Pic (X)$ is given by $2H-3\l, H-\l$ with intersection pairing 
$$(2H-3\l)^2 = 2, \quad (H-\l)^2 = 4, \quad (2H-3\l)\cdot (H-\l) = 5.$$
If we try to find a divisor $C=aH+b\l$ such that 
$(2H-3\l)\cdot C=5$ and $C^2=4$, there are two possibilities for $C$ given
by $H-\l$ and $9H-14\l$.
This gives us two corresponding possibilities 
$$(1,-H+2\l,-4)+\Z v, \mbox{ and } 
(-6,-2H-2\l,-11)+\Z v$$
in $v^\perp/\Z v $.  We present the analysis considering this first solution
and skip the other similar case.
So we get the following information about $f_\E$ under our choices.
\begin{eqnarray*}
f_\E: \tilde{\H}^{1,1}(X, \Z) & \rightarrow & \tilde{\H}^{1,1}(M, \Z) \\
 (2,H,2) & \mapsto & (0,0,1)\\
(1,0,-1) + x\cdot v &\mapsto & (0,2H-3l,0) \\
(1,-H+2\l,-4) + y \cdot v &\mapsto & (0,H-l,0)\\
\end{eqnarray*}
We write a matrix for $f^{-1}_\E$ in terms of the basis $$(1,0,0),(0,H,0),
(0,\l,0),(0,0,1).$$
Using the above equations we solve for 
$$f^{-1}_\E(0,H,0)=(1,-2H+4\l,7)+sv$$ $$f^{-1}_\E(0,\l,0)=(2,-3H+6\l,-11)+tv$$
and we get the partially defined matrix for $f^{-1}_\E$
$$P=\begin{pmatrix}
a & 2+2s & 1+2t & 2 \\
b & -3+s & -2+t & 1 \\
c & 6 & 4 & 0 \\
d & -11+2s & -7+2t & 2 
\end{pmatrix}.$$
Now if we let $A$ be the Gram matrix of the Mukai paring in terms of this
basis we have that 
$$A =\begin{pmatrix}
0 & 0 & 0 & -1 \\
0 & 8 & 1 & 0 \\
0 & 1 & -2 & 0 \\
-1 & 0 & 0 & 0 
\end{pmatrix}.$$
Since $f^{-1}_\E$ is an isometry, we get the matrix equation
$P^T A P = A$ which gives us the four nontrivial polynomial equations
for the variables $a,b,c,d,s,t.$  This includes one linear equation
which turns out to be $\det P =2d-8b-c+2a=1$.
Solving this for $c$ and substituting into the other equations yields
two linear equations
$$ 21d+13a-68b+t=10$$
$$ 32d+19a-102b+s=15.$$
We solve these for $s$ and $t$.  
Our matrix now depends on $a,b,d$ which satisfy the quadratic equation:
$$q:=-18da+68bd+68ba-34b-136b^2-8d^2+8d-8a^2+8a-2=0$$
We can now compute the inverse
of $P$  to obtain a matrix for $f_\E$ 
and observe that its last column shows that 
$v(E_x)=(2,(6a-32b+10d-5)H+(-10a+52b-16d+8)\l,a)$
which we can twist by $(3-5d+16b-3a)H+(5a-26b+8d-4)\l$
to show that
$v(E_x) = (2,H,2)$ up to equivalence.  
\end{proof}

We are now in a position to prove our main theorem about the classical moduli problem.

\begin{theorem}\label{thm:main} 
Let $X$ be a K3 surface of degree $8$ in $\P^5.$  
Assume that $X$ contains 
a line $\l$ and $\rho (X) = 2.$  
Then $X$ is the base locus of a net of quadrics.  Let $\phi:M \rightarrow \P^2$ be the double cover of $\P^2$ branched along the sextic curve parameterising degenerate quadrics in the net. Then
\begin{enumerate}
\item The moduli space $\M_H(2,H,2) \simeq M$ is a fine moduli space and  is  isomorphic to $X$.
\item  The universal sheaf $\E$ on $X \times X$ is symmetric, i,e, its restriction 
to either of the two factors is a rank $2$ sheaf with Mukai vector $(2,H,2)$ 
up to twists by line bundles. 
\end{enumerate}
\end{theorem}

\begin{proof}

Recall from Section~\ref{sec:rk2vb} that there exists a conic bundle $I$ on 
$X \times M$ such that $I_{|_{X \times \{m\}}}$ is a rank $2$ vector bundle with Mukai vector 
$(2,H,2).$
We also showed earlier that in this case (i.e. when $X$ contains a line) the conic bundle $I$ admits a section and hence lifts to a rank two universal sheaf $\cE$ on $X \times M.$  
 Also $\sigma = 1$ because for $v = (2,H,2)$ and $w = (0,\l,0)$ the Mukai pairing gives $(v,w) = 1$ which implies $\M_H(2,H,2)$ is a fine moduli space.

We already know from Theorem~\ref{thm:M-nonempty} that  $\M_H(2,H,2)$ is a K3 surface which is birational to $M.$  Since $M \simeq X$ (\cite{MadonnaNikulin}) it follows that 
$\M_H(2,H,2) \simeq M \simeq X.$  For clarity of notation we continue to denote 
$\M_H(2,H,2)$ by $M$ instead of $X$ even though they are isomorphic.  

Since $M$ is a fine moduli space, Mukai's results (Theorem~\ref{thm:mukai}) show that $f_\cE$ induces a Hodge isometry of the lattices
$$f_\cE : \tilde{\H}(X, \Z) \rightarrow \tilde{\H}(M, \Z).$$  
Let $E_x := \cE_{|_{\{x\} \times M}}.$  Then $v(E_x)$ is also an isotropic element of 
$\tilde{\H}^{1,1}(M, \Z).$  
In fact one proves $v(E_x) = f_\E(0,0,1)$ using the same arguments in Lemma~4.11~\cite{Mukai2} and the Grothendieck-Riemann-Roch Theorem. Since $f_\E$ is a Hodge isometry, $v(E_x)$ is primitive.  

By Theorem~\ref{thm:mukai-map} it follows that $f_\E(0,0,1) = (2,H,2)$ up to equivalence.
So we get a two dimensional flat family of sheaves parameterised by $X$ with 
Mukai vector $w \simeq (2,H,2).$
\end{proof}

So far we have considered examples $X$ with $\rho (X) = 2$ and $M$ is a fine moduli space.
However it can happen that $\rho (X) = 2$ but $M$ is not a fine moduli space.
 Given a smooth plane sextic curve $C$
and a choice of an even, ineffective theta-characteristic $L$ on $C$, 
there exists a family of quadrics $\mathcal{Q}$ in $\P^5$ such that $V(\det \mathcal{Q})=C$.
So there are as many nets of quadrics $\mathcal{Q}$ as there are 
theta-characteristics $L$ on $C$ with $h^0(L)=0$ (Theorem 1 \cite{Tjurin1}).  
For a generic 
curve of genus $g$ the number of even, ineffective theta characteristics is given by $2^{g-1}(2^{g} + 1).$  So in our case there exist $2^9(2^{10} + 1)$ such nets of quadrics for a generic $C.$  In \cite{Geemen} it is shown that there exist $2^9(2^{10}+1)$ distinct even sublattices of $T_M$ of index 2.  These correspond to certain two-torsion elements of $\Br(M).$  
Another interesting relation is that 
the set of theta-characteristics on a curve $C$ of genus $g$ is in 
one to one correspondence with the set of spin structures on $C$
as in \cite{Atiyah}~Proposition~3.2.

Returning to the case where $X$ contains a conic, we have the following result.

\begin{theorem}\label{thm:M-nonfine} 
Let $X$ be a K3 surface of degree 8 in $\P^5$ which contains a plane
conic and has $\rho(X)=2$. Then the moduli space $\cM(2,H,2)$ is isomorphic to $M$ and is non fine.
\end{theorem} 
\begin{proof}
Since $\Pic(X)=\Z \oplus \Z$, generated by the conic $C$ and $H$, we
may let $f=aC+bH$ and we see that there are no rational solutions to
$f^2=0$ and $f.H=3$ or $f.H=4$.  So $X$ is a complete intersection by 
Proposition~\ref{compint} and $\M$ is compact by Theorem~\ref{thm:M-nonempty}.
The same theorem also shows that $\M$ is a minimal resolution of $M$. Recall that the 
sextic branch locus admits a tri-tangent.
The Picard group of $M$ is generated by $h := \phi^{*}(\O_{\P^2}(1))$ and the
curve $\Gamma$, where
$\Gamma+\Gamma' := \phi^{*}(\l)$ are the components of $M$
over the tritangent.  If $M$ is singular then $\M$ contains an
independent $(-2)$-curve so we have the contradiction that 
$2=\rho(X) = \rho(\M) \geq 3$.  So $\M$ is isomorphic to $M$.

  The intersection numbers on $M$
are $h^2 = 2, h . \Gamma = 1, {\Gamma}^2 = -2.$
Suppose for contradiction that $M$ is a fine moduli space.  Then there exists a universal sheaf $\E$ on $X \times M,$ and a Hodge isomorphism of lattices
$$f_\E : \tilde{\H}(X, \Z) \rightarrow \tilde{\H}(M, \Z).$$
So $f_\E(0,0,1) = v(\E_{|_{\{x\} \times M}}) := v(E_x)$ is an isotropic element of $\tilde{\H}(M, \Z).$  Then $$v(E_x) = (2, m h + n \Gamma, k)$$
where we require that 
$(m  h + n  \Gamma)^2 - 4k = 0.$  
Since the restriction of $E_x$ to $h$, a generic element of $|h|,$ is a vector bundle of odd degree (\cite{Newstead}), we see that 
$c_1(E_x).h = 2m + n$ is an odd number and hence $n$ is odd.
But this is impossible since we have
$$ m^2 + mn - n^2 = 2k,$$
and for all values of $m$ the right hand side is odd.
So $M$ is not a fine moduli space. 
\end{proof}

Most of our results in this paper deal with K3 surfaces satisfying $\rho(X)=2$.  
We are investigating stability conditions of this classical problem in other contexts such as $\rho(X) > 2$.  In addition we have geometric constructions for the associated Azumaya algebras in the cases when $\M_H(2,H,2)$ is a non-fine moduli space.The results will be published in a forthcoming paper \cite{IngallsKhalid}.

\end{section}

\end{document}